\documentclass[10pt,reqno]{amsart}
\usepackage{amsmath,amssymb,latexsym,esint,cite,mathrsfs,dsfont}
\usepackage{verbatim,cite,relsize,color,soul}
\usepackage[latin1]{inputenc}
\usepackage{microtype}
\usepackage{color,enumitem,graphicx}
\usepackage[colorlinks=true,urlcolor=blue, citecolor=red,linkcolor=blue,
linktocpage,pdfpagelabels, bookmarksnumbered,bookmarksopen]{hyperref}
\usepackage[hyperpageref]{backref}
\usepackage[english]{babel}
\usepackage{tikz}
\usepackage[font=small,skip=5pt]{caption}
\usepackage{enumitem}

\usepackage{amsrefs}

\usepackage{geometry}
\numberwithin{equation}{section}

\newtheorem{theorem}{Theorem}[section]
\newtheorem{lemma}[theorem]{Lemma}

\newtheorem*{theorem*}{Conjecture}

\newtheoremstyle{remarkstyle}
{}{}{}{ }{\bfseries}{.}{ }{\thmname{#1}\thmnumber{ #2}\thmnote{ (#3)}}
\theoremstyle{remarkstyle}
\newtheorem{remark}{Remark}[section]

\newcommand{\R}{\mathbb R}
\newcommand{\C}{\mathbb C}
\newcommand{\Ac}{\mathcal A}
\newcommand{\Bc}{\mathcal B}

\newcommand{\Mcal}{\mathcal M}
\newcommand{\vareps}{\varepsilon}

\DeclareMathOperator*{\opt}{opt}
\DeclareMathOperator*{\GN}{GN}
\DeclareMathOperator*{\Sob}{Sob}

\DeclareMathOperator*{\rea}{Re}
\DeclareMathOperator*{\ima}{Im}
\DeclareMathOperator*{\gamc}{{\gamma_c}}

\DeclareMathOperator*{\cb}{{\boldsymbol{c}}}

\title[Blow-up for System NLS]
{Blow-up results for systems of nonlinear Schr\"odinger equations with quadratic interaction}

\author[V. D. Dinh]{Van Duong Dinh}
\address[V. D. Dinh]{Laboratoire Paul Painlev\'e UMR 8524, Universit\'e de Lille CNRS, 59655 Villeneuve d'Ascq Cedex, France
and 
Department of Mathematics, HCMC University of Pedagogy, 280 An Duong Vuong, Ho Chi Minh, Vietnam}
\email{contact@duongdinh.com}

\author[L. Forcella]{Luigi Forcella}
\address[L. Forcella]{Department of Mathematics, Heriot-Watt University and The Maxwell Institute for the Mathematical Sciences, Edinburgh, EH14 4AS, United Kingdom}
\email{l.forcella@hw.ac.uk}


\begin{document}
	
	\begin{abstract}
		We establish blow-up results for systems of NLS equations with quadratic interaction in anisotropic spaces. We precisely show finite time blow-up or grow-up for cylindrical symmetric solutions. With our construction, we moreover prove some polynomial lower bounds on the kinetic energy of global solutions in the mass-critical case, which in turn implies grow-up along any diverging time sequence. Our analysis extends to general NLS systems with quadratic interactions, and it also provides improvements of known results in the radial case. 
	\end{abstract}
	
	\maketitle

	\section{Introduction}
	\label{S1}
	\setcounter{equation}{0}
	
In this paper, we investigate  the existence of blowing-up solutions for the Cauchy problem for the following system of nonlinear Schr\"odinger equations with quadratic interaction
	\begin{align} \label{QNLS}
	\left\{
	\renewcommand*{\arraystretch}{1.2}
	\begin{array}{rcl}
	i\partial_t u + \frac{1}{2m}\Delta u &=& \lambda v \overline{u}, \\
	i \partial_t v + \frac{1}{2M} \Delta v &=& \mu u^2, 
	\end{array}
	\right.	
	\end{align}
	where the wave functions $u,v: \R \times \R^d \rightarrow \C$ are complex scalar functions, the parameters $m, M$ are two real positive quantities, and $\lambda, \mu \in \C$ are two complex coupling constants.  

Multi-components systems of nonlinear Schr\"odinger equations with quadratic-type interactions appear in the processes of waves propagation in quadratic media. They model, for example, the Raman amplification phenomena in a plasma, or they are used to describe other phenomena in nonlinear optics. We refer the readers to \cite{CCO, CdMS, Kiv, KS} for more insights on these kind of physical models.

In the case of the so-called mass-resonance condition, namely provided that the condition
 	\begin{align} \label{mass-reso}
	M=2m
	\end{align} 
is satisfied, the system \eqref{QNLS} can be viewed, see \cite{HOT}, as a non-relativistic limit of the following system of nonlinear Klein-Gordon equations
	\[
	\left\{
	\renewcommand*{\arraystretch}{1.2}
	\begin{array}{rcl}
	\frac{1}{2c^2 m} \partial^2_t u -\frac{1}{2m} \Delta u + \frac{mc^2}{2} u &=& -\lambda v \overline{u}, \\
	\frac{1}{2c^2 M} \partial^2_t v -\frac{1}{2M} \Delta v +\frac{Mc^2}{2} v &=& - \mu u^2, \\
	\end{array}
	\right.	
	\]
as the speed of light $c$ tends to infinity. \\

	To the best of our knowledge, the first mathematical study of the system \eqref{QNLS} is due to Hayashi, Ozawa, and Tanaka \cite{HOT}, where, among other things, they established the local well-posedness of the system \eqref{QNLS}, and they proved that, in order to ensure the conservation law of the total charge, namely the sum (up to some constant) of the $L^2$ norm of $u$ and $v,$ it is natural to consider the condition 
	\begin{align} \label{cond-lambda-mu}
	\exists\, c \in \R \backslash \{0\} \ \hbox{ such that } \ \lambda = c \overline{\mu}.
	\end{align}
	
\noindent	Moreover, if we assume that $\lambda, \mu$ satisfy \eqref{cond-lambda-mu} for some $c>0$ and $\lambda, \mu \ne 0,$ by the change of variable
	\[
	\tilde{u}(t,x) = \sqrt{\frac{c}{2}} |\mu| u \left(t, \sqrt{\frac{1}{2m}} x\right), \quad \tilde{v}(t,x) = -\frac{\lambda}{2} v \left(t,\sqrt{\frac{1}{2m}} x\right), 
	\]
	the system \eqref{QNLS} can be written (by dropping the tildes) as
	 \begin{align} \label{SNLS}
	 \left\{
	 \begin{array}{rcl}
	 i\partial_t u + \Delta u &=& - 2 v \overline{u}, \\
	 i \partial_t v + \kappa \Delta v &=& - u^2,
	 \end{array}
	 \right.	
	 \end{align}
	 where $\kappa =\frac{m}{M}$ is the mass ratio. Note that $\kappa=\frac{1}{2}$ in the mass-resonance case \eqref{mass-reso}. The system \eqref{SNLS} satisfies the conservation of mass and energy defined respectively by
	 \begin{align*} 
	  M(u(t),v(t)) &= \|u(t)\|^2_{L^2} + 2 \|v(t)\|^2_{L^2}, \\
	 E(u(t),v(t)) &= \frac{1}{2} \|\nabla u(t)\|^2_{L^2} + \frac{\kappa}{2} \|\nabla v(t)\|^2_{L^2} -  \rea \int v(t) \overline{u}^2(t) dx.
	 \end{align*}
	For the purpose of our paper, we define the \emph{kinetic} energy 
	 \begin{equation}\label{defi-K}
	 T(f,g) := \|\nabla f\|^2_{L^2} + \kappa \|\nabla g\|^2_{L^2},
	 \end{equation}
	 and the \emph{potential} energy by
	 \begin{equation}\label{defi-N}
	 P(f,g) := \rea \int g \overline{f}^2 dx,
	 \end{equation}
hence we rewrite the total energy as
\[
E(u(t),v(t))= \frac12 T(u(t),v(t))- P(u(t),v(t)).
\] 

\noindent We also introduce the following functional defined in terms of $T$ and $P:$
\begin{align} \label{defi-G}
	G(f,g):= T(f,g) - \frac{d}{2} P(f,g).
	\end{align}
Even if the we will use $G$ evaluated at time-dependent solutions, it is worth mentioning that $G$ is the Pohozaev functional which is strictly related to the time-independent elliptic equations \eqref{ground-1} and \eqref{ground-2} below.

Another crucial property of \eqref{SNLS} is that \eqref{SNLS} is invariant under the scaling
	 	\begin{align} \label{scaling}
	 	u_\lambda(t,x):= \lambda^2 u(\lambda^2 t, \lambda x), \quad v_\lambda(t,x):= \lambda^2 v(\lambda^2 t, \lambda x), \quad \lambda>0.
	 	\end{align}
A direct computation gives
	 	\begin{align*}
	 	\|u_\lambda(0)\|_{\dot{H}^\gamma} = \lambda^{\gamma-\frac{d}{2} +2} \|u(0)\|_{\dot{H}^\gamma}, \quad \|v_\lambda(0)\|_{\dot{H}^\gamma} = \lambda^{\gamma-\frac{d}{2} +2} \|v(0)\|_{\dot{H}^\gamma}.
	 	\end{align*}
	 	This shows that \eqref{scaling} leaves the $\dot{H}^{\gamc}$-norm of initial data invariant, where
	 	\[
	 	\gamc:= \frac{d}{2}-2.
	 	\]
	 	According to the conservation laws of mass and energy, \eqref{SNLS} is called mass-critical, mass and energy intercritical (or intercritical for short), and energy-critical if $d=4$, $d=5$, and $d=6$, respectively.\\

In the present paper, we restrict our attention to the dimensions $d=4,5,6,$ and we are interested in showing the formation of singularities in finite or infinite time for solutions to the initial value problem associated to \eqref{SNLS}, with initial data 
\[
(u,v)(0,\cdot)=:(u_0,v_0)\in H^1(\R^d)\times H^1(\R^d).
\] 
As well-known, the existence of blowing-up solutions to the Schr\"odinger-type equations is closely related to the notion of standing wave or static (in the energy-critical case) solutions. Therefore, before stating our main results, we recall some basic facts about the existence of ground states for \eqref{SNLS}.\\

First of all, we recall that by standing waves solutions we mean solutions to \eqref{SNLS} of the form 
\[
	(u(t,x),v(t,x)) = (e^{it} \phi(x), e^{2it} \psi(x)),
	\]
	where $\phi, \psi$ are real-valued functions satisfying
	\begin{align} \label{ground-1}
	\left\{
	\begin{array}{ccl}
	- \Delta \phi + \phi &=&  2 \phi \psi, \\
	- \kappa \Delta \psi + 2\psi&=& \phi^2.
	\end{array}
	\right.	
	\end{align}

In \cite{HOT}, Hayashi, Ozawa, and Tanaka showed the existence of ground states related to \eqref{ground-1}, i.e. non-trivial solutions to \eqref{ground-1} that minimizes the action functional
	\[
	S(f,g):= E(f,g) +\frac{1}{2} M(f,g)
	\]
	over all non-trivial solutions to \eqref{ground-1}. It is worth mentioning that this existence result holds whenever $d\leq5,$ and not only for $d=4,5.$ When $d=6$, i.e. the energy-critical case, \eqref{SNLS} admits a static solution of the form
	\[
	(u(t,x),v(t,x)) = (\phi(x), \psi(x)),
	\]
	where $\phi,\psi$ are real-valued functions satisfying
	\begin{align} \label{ground-2}
	\left\{
	\begin{array}{ccl}
	- \Delta \phi  &=&  2 \phi \psi, \\
	- \kappa \Delta \psi &=& \phi^2.
	\end{array}
	\right.	
	\end{align}
	The existence of ground states related to \eqref{ground-2} was shown in \cite{HOT} (see also \cite[Section 3]{NP-blow}). Here by a ground state related to \eqref{ground-2}, we mean a non-trivial solution to \eqref{ground-2} that minimizes the energy functional over all non-trivial solutions of \eqref{ground-2}.

	\section{Main results}\label{S-Main} 
	 \label{S2}
	 \setcounter{equation}{0}
	 We are now ready to state our first result  about the blow-up of solutions in the mass and energy intercritical case in anisotropic spaces. To this aim, we introduce some notation. Denote
	 \begin{align} \label{Sigma-d}
	 		\Sigma_d:= \left\{ f \in H^1 (\R^d) \ \hbox{ s.t. } f(y,x_d) = f(|y|,x_d), ~ x_d f \in L^2(\R^d)\right\},
	 		\end{align}
where $x=(y,x_d), y=(x_1, \dots, x_{d-1}) \in \R^{d-1}$, and $x_d \in \R$. Here $\Sigma_d$ stands for the space of cylindrical symmetric functions with finite variance in the last direction. We also introduce the following blow-up conditions:
\begin{align} \label{cond-blow}\tag{BC$_{5d}$}
	E(u_0,v_0) M(u_0,v_0) < E(\phi,\psi) M(\phi,\psi) \quad \& \quad T(u_0,v_0) M(u_0,v_0) > T(\phi,\psi) M(\phi,\psi).
	\end{align}
As for the usual Schr\"odinger equation, the conditions expressed in \eqref{cond-blow} are the counterpart of  the conditions
\begin{align} \label{cond-gwp}\tag{SC$_{5d}$}
	E(u_0,v_0) M(u_0,v_0) < E(\phi,\psi) M(\phi,\psi) \quad \& \quad T(u_0,v_0) M(u_0,v_0) < T(\phi,\psi) M(\phi,\psi),
	\end{align} 
in the dichotomy leading to global well-posedness \& scattering (\eqref{cond-gwp}) or blow-up (\eqref{cond-blow}). In the energy critical case, the previous conditions in \eqref{cond-blow} will be replaced by analogous inequalities, see \eqref{cond-blow-6d} below. Since in this paper we are concerned only with the blow-up dynamics of solutions to \eqref{SNLS}, we will not use the modified conditions for the scattering theory.

\subsection{Intercritical case} 
Our first result concerns a finite time blow-up for \eqref{SNLS} in the intercritical case $d=5$.
	 \begin{theorem} \label{theo-blow-5d}
	 	Let $d=5$, $\kappa >0$, and $(\phi,\psi)$ be a ground state related to \eqref{ground-1}. Let $(u_0,v_0) \in \Sigma_5 \times \Sigma_5$ satisfy  \eqref{cond-blow}.  Then the corresponding solution to \eqref{SNLS} blows-up in finite time.
	 \end{theorem}
Let us give some comments on the previously known blow-up results for the system \eqref{SNLS}. The formation of singularities in finite time for negative energy and radial data was shown by Yoshida in \cite{Yoshida}, while for non-negative energy radial data a proof was recently given by Inui, Kishimoto, and Nishimura. Specifically, they proved in \cite{IKN-NA} the blow-up for radial initial data satisfying 
\begin{equation} \label{cond-IKN-5d}
E(u_0,v_0) M(u_0,v_0)< E(\phi,\psi) M(\phi,\psi) \quad\& \quad G(u_0,v_0) <0.
\end{equation}
By a variational characterization, we show in Lemma \ref{lem-equi-5d} that \eqref{cond-blow} and \eqref{cond-IKN-5d} are indeed equivalent. Thus a version of Theorem \ref{theo-blow-5d} for radial solutions  would be an interchangeable restatement of the result obtained in \cite{IKN-NA}. \\

Despite our approach relies on the classical virial identities, we need to precisely construct suitable cylindrical cut-off functions enabling us to get  enough decay (by means of some Sobolev embedding for partially radial functions) to close our estimates. With respect to the classical NLS equation, we will use an ODE argument instead of a concavity argument to prove our results, by only using the first derivative in time of suitable localized quantity, see Section \ref{S3}. We refer the reader to the early work of Martel \cite{Mar} in the context of the NLS equation in anisotropic spaces, and the more recent papers \cite{BF20, Inui1, Inui2}. See also our recent paper  \cite{AVDF} in the context of NLS system with cubic interaction.\\

For sake of completeness, we report now  known blow-up  and long time dynamics  results for  \eqref{SNLS} in the intercritical case.\\

If $\kappa =\frac{1}{2}$, Hayashi, Ozawa, and Tanaka in \cite{HOT} showed a blow-up result with negative energy and finite variance data, i.e. initial data belonging to $\Sigma\times\Sigma:=(H^1\times H^1)\cap (L^2(|x|^2\, dx) \times L^2(|x|^2\, dx))$. Hamano, see \cite{Hamano}, proved the scattering below the mass energy ground state. More precisely, he proved that if $(u_0, v_0) \in H^1 \times H^1$ satisfies \eqref{cond-gwp}  then the corresponding solution to \eqref{SNLS} exists globally in time and scatters in $H^1\times H^1$ in both directions, i.e. there exist $(u_\pm, v_\pm) \in H^1\times H^1$ such that
\[
\|(u(t),v(t)) - (e^{it\Delta} u_\pm, e^{i\kappa t\Delta} v_\pm)\|_{H^1\times H^1} \rightarrow 0
\]
as $t\rightarrow \pm \infty$. Here $e^{itc\Delta}$ denotes the classical Schr\"odinger free propagator. In addition, if $(u_0,v_0) \in H^1 \times H^1$ satisfies \eqref{cond-blow}, then the corresponding solution to \eqref{SNLS} either blows-up in finite time or there exists $|t_n|\rightarrow \infty$ such that $\|(u(t_n),v(t_n))\|_{H^1\times H^1} \rightarrow \infty$ as $n\rightarrow \infty$. Furthermore, if $(u_0,v_0) \in \Sigma \times \Sigma$ or $(u_0,v_0)$ is radial, then the solution blows-up in finite time. The first author, see \cite{Dinh-insta}, established the strong instability by blow-up for ground state standing waves of \eqref{SNLS}.
	 
If $\kappa \ne \frac{1}{2}$, Hamano, Inui, and Nishimura \cite{HIN} established the scattering for radial data below the mass-energy threshold. The proof is based on the concentration/compactness and rigidity scheme in the spirit of Kenig and Merle \cite{KM}. Wang and Yang \cite{WY}  extended the result of \cite{HIN} to the non-radial case provided that $\kappa$ belongs to a small neighbourhood  of $\frac{1}{2}$. Their proof made use of a recent method of Dodson and Murphy \cite{DM-MRL} using the interaction Morawetz inequality. Noguera and Pastor \cite{NP-DPDE} proved that if $(u_0, v_0) \in H^1 \times H^1$ satisfies \eqref{cond-gwp}, then the corresponding solution to \eqref{SNLS} exists globally in time. 
	 
\begin{remark}
From a pure mathematical perspective, distinguishing  the cases $\kappa=\frac12$ and $0<\kappa\neq\frac12$ plays a role in the virial identities related to \eqref{SNLS}. Under the mass-resonance condition, namely $\kappa=\frac12,$ some terms in the virial identities disappear, and the study of the dynamics of solutions is easier due to these cancellations. This is no more the case in the non-mass-resonance setting, i.e. when $\kappa\neq\frac12.$  We refer the reader to \cite[Introduction]{NP-blow} for an exhaustive list of references in which the effects of the mass and non-mass resonance conditions on the dynamics of solutions to systems similar to \eqref{SNLS} were studied.
\end{remark}

	 
\subsection{Energy-critical case}
Our next Theorem deals with a blow-up result in the energy-critical case $d=6.$
	 
	 \begin{theorem} \label{theo-blow-6d}
	 	Let $d=6$, $\kappa >0$, and $(\phi,\psi)$ be a ground state related to \eqref{ground-2}. Let $(u_0,v_0) \in  \Sigma_6 \times \Sigma_6$ satisfy 
			\begin{align}\tag{BC$_{6d}$}\label{cond-blow-6d}
	 	E(u_0,v_0) < E(\phi,\psi), \quad T(u_0,v_0)> T(\phi,\psi).
	 	\end{align} 
	 Then the corresponding solution to \eqref{SNLS} blows-up in finite time.
	 \end{theorem}
	 
It is worth mentioning that finite time blow-up with negative energy radial data was established in \cite{Yoshida}, while for non-negative energy radial data, the blow-up result was shown in \cite{IKN-NA} for data satisfying
	 		\begin{align} \label{cond-IKN-6d}
	 		E(u_0,v_0) < E(\phi,\psi), \quad G(u_0,v_0) <0.
	 		\end{align} 
Since we will prove in Lemma \ref{lem-equi-6d} that \eqref{cond-blow-6d} is equivalent to \eqref{cond-IKN-6d}, our result restricted to a radial framework, would be equivalent to the one in \cite{IKN-NA}.

\begin{remark}
If $\kappa =\frac{1}{2}$,  the blow-up result with negative energy and finite variance data was shown in Hayashi, Ozawa, and Tanaka, see \cite{HOT}. 
\end{remark}
	 	
\subsection{Mass-critical case}

In the mass-critical case, we have the following blow-up or grow-up results for \eqref{SNLS}.
	 
	 \begin{theorem} \label{theo-blow-grow-4d}
	 	Let $d=4$ and $0<\kappa \ne \frac{1}{2}$. Let $(u_0,v_0) \in H^1 \times H^1$ be radially symmetric satisfying $E(u_0,v_0)<0$. Then the corresponding solution to \eqref{SNLS} either blows-up forward in finite time, i.e. $T^*<\infty$, or it blows-up in infinite time in the sense that $T^*=\infty$ and 
	 	\begin{equation}\label{grow-4d}
	 	T(u(t),v(t)) \geq C t^2
	 	\end{equation}
	 	for all $t\geq t_0$, where $C>0$ and $t_0 \gg 1$ depend only on $\kappa, M(u_0,v_0)$, and $E(u_0,v_0)$. A similar statement holds for negative times.
	 \end{theorem}
Under the assumption of Theorem \ref{theo-blow-grow-4d}, the blow-up or grow-up result along one time sequence was proved in \cite[Theorem 1.2]{IKN-NA}. More precisely, if $T^*=\infty$, then there exists a time sequence $t_n\rightarrow \infty$ such that $\|(u(t_n),v(t_n))\|_{H^1\times H^1}\rightarrow \infty$ as $n\rightarrow \infty$. 

By performing a more careful analysis, our argument yields to a stronger result with respect to the one in \cite{IKN-NA}. Indeed, we are able to show a growth rate for the kinetic energy of the form \eqref{grow-4d} which in turn implies the grow-up result along an arbitrary diverging sequence of times.
We would like to mention that this grow-up result along any diverging time sequence, is also an interesting open problem related to the usual mass-supercritical focusing cubic 3D NLS, see the weak conjecture of Holmer and Roudenko in \cite{HR-CPDE}.
	 
\begin{remark} In the case $\kappa=\frac{1}{2}$ and for radial data with negative energy, the finite time blow-up was shown by the first author in \cite{Dinh-NA}. For the long time dynamics in the mass-critical case we refer to \cite{IKN-mass}.
\end{remark}
	 
We  give now the following blow-up or grow-up result for anisotropic solutions to \eqref{SNLS}.
	 
\begin{theorem} \label{theo-blow-grow-Sigma-4d}
	 	Let $d=4$ and $0<\kappa \ne \frac{1}{2}$. Let $(u_0,v_0) \in \Sigma_4 \times \Sigma_4$ satisfy $E(u_0,v_0)<0$. Then the corresponding solution to \eqref{SNLS} either blows-up forward in finite time, i.e. $T^*<\infty$, or $T^*=\infty$ and there exists a time sequence $t_n\rightarrow \infty$ such that $\|(u(t_n),v(t_n))\|_{H^1 \times H^1} \rightarrow \infty$ as $n\rightarrow \infty$. If we assume $\kappa =\frac{1}{2}$, then either $T^*<\infty$ or $T^*=\infty$ and there exists a time sequence $t_n\rightarrow \infty$ such that $\|\partial_4u(t_n)\|_{L^2} \rightarrow \infty$ as $n\rightarrow \infty$. A similar statement holds for negative times.
\end{theorem}

\subsection{Extensions to a general system of NLS with quadratic interactions}
We conclude this section by listing some extensions of the previous Theorems for general NLS systems with quadratic interactions.

In dimension $d=5$ and $d=6,$ namely in the mass-supercritical and the energy-critical case, respectively, the results above can be extended -- provided that some structural hypothesis are satisfied -- to the following initial value problem for general system of NLS with quadratic interactions:
\begin{align} \label{GQNLS}
	\left\{
	\renewcommand*{\arraystretch}{1.2}
	\begin{array}{rcl}
	ia_j\partial_t u_j+ b_j\Delta u_j -c_ju_j &=&-f_j(u_1,\dots, u_N), \quad j\in\{1,\dots, N\},\\
	(u_1,\dots, u_N)(0,\cdot) &=& (u_{0,1},\dots,u_{0,N}) \in H^1(\R^d) \times \dots \times H^1(\R^d), 
	\end{array}
	\right.	
	\end{align}
where $u_j:\R\times\R^d\to\C$, the parameters $a_j,b_j,c_j$ are real coefficients satisfying $a_j>0, b_j>0$ and $c_j\geq0$, and the functions $f_j$ grow quadratically for all $j =1,\dots, N$. More precisely, under the assumptions (H1)--(H8) in \cite{NP-blow}, Theorems \ref{theo-blow-5d} and \ref{theo-blow-6d} can be stated for \eqref{GQNLS} as well, with the necessary modifications. In particular, the set of conditions (H1)--(H8) in \cite{NP-blow} (see also \cite{NP-CCM}) ensure that \eqref{GQNLS} is local well-posed, there exist ground states (along with stability and instability properties), and the mass and the energy are conserved. Here the mass is defined by
\[
\mathcal M(\vec u(t)):=\sum_{j=1}^{N}\frac{a_j s_j}{2}\|u_j(t)\|_{L^2}^2,
\]
where the real parameters $s_j>0$ satisfy 
\[
\ima \sum_{j=1}^{N}s_jf_j(\boldsymbol{z})\bar z_j=0, \quad \forall \boldsymbol{z}=(z_1,\dots, z_N)\in\C^N,
\] 
and $\vec u$ is  the compact notation for $(u_1,\dots, u_N)$. The energy is instead defined by 
\[
\mathcal E(\vec u(t)):= \frac{1}{2}\sum_{j=1}^{N} b_j\|\nabla u_j(t)\|_{L^2}^2 +\frac{1}{2}\sum_{j=1}^{N}c_j \|u_j(t)\|_{L^2}^2- \rea \int F(\vec u(t)) dx,
\]
where $F:\C^N\to\C$ is such that $f_j=\partial_{\bar z_j}F+\overline{\partial_{z_j}F}$ for any $j\in\{1,\dots, N\}.$\\

In $d=5$, we denote by $\vec\phi=(\phi_1,\dots, \phi_N)$ the ground state related to the system of elliptic equations
\begin{align} \label{ell-equ-vec}
-b_j\Delta \phi_j+\left(\frac{a_js_j}{2}\omega+c_j \right)\phi_j=f_j(\vec \phi), \quad j\in\{1,\dots, N\},\quad \omega\in\R,
\end{align}
i.e. $\vec{\phi}$ is a non-trivial real-valued solution of \eqref{ell-equ-vec} that minimizes the action functional
\[
\mathcal{S}(\vec{\phi}) = \frac{1}{2} \mathcal{T}(\vec{\phi}) + \frac{1}{2} \mathcal{Q}(\vec{\phi}) - \mathcal{P}(\vec{\phi})
\]
over all non-trivial real-valued solutions to \eqref{ell-equ-vec}, where
\[
\mathcal T(\vec g):=\sum_{j=1}^Nb_j\|\nabla g_j\|_{L^2}^2, \quad \mathcal Q(\vec g):=\sum_{j=1}^N\left(\frac{a_js_j}{2}\omega+c_j\right)\|g_j\|_{L^2}^2, \quad \mathcal P (\vec g) := \rea \int F(\vec g) dx.
\]
Under the assumptions (H1)--(H8) in \cite{NP-blow} , ground states related to \eqref{ell-equ-vec} do exist if 
\begin{align} \label{cond-c}
\frac{a_js_j}{2}\omega+c_j >0, \quad \forall j \in \{1, \dots, N\}. 
\end{align}
If we denote by $\mathbb G(\omega, \cb)$ the set of ground states related to \eqref{ell-equ-vec}, where $\cb=(c_1, \dots, c_N)$, then $\mathbb G(\omega, \cb) \ne \emptyset$ provided that \eqref{cond-c} is satisfied. In particular, $\mathbb G (1,\boldsymbol{0}) \ne \emptyset$. Moreover, the following Gagliardo-Nirenberg inequality
\begin{align} \label{GN-ineq-vec}
P(\vec g) \leq C_{\opt} \left[ \mathcal Q(\vec g) \right]^{\frac{6-d}{4}} \left[ \mathcal T (\vec g) \right]^{\frac{d}{4}}
\end{align}
is achieved by a ground state $\vec \phi \in \mathbb G (\omega, \cb)$. We refer the reader to \cite[Section 4]{NP-CCM} for more details on ground states related to \eqref{ell-equ-vec}.

 By adapting the arguments presented in this paper, we can prove the result in Theorem \ref{theo-blow-5d} provided that we replace \eqref{cond-blow} with 
\begin{equation}\tag{BC$^\prime_{5d}$}
\mathcal M(\vec u_0) \mathcal E(\vec u_0)<\mathcal M(\vec \phi)\mathcal E_0(\vec \phi) \quad \& \quad \mathcal M(\vec u_0)\mathcal T(\vec u_0)>\mathcal M(\vec \phi)\mathcal T(\vec \phi),
\end{equation}
where $\vec \phi \in \mathbb G(1,\boldsymbol{0})$ and 
\[
\mathcal E_0(\vec g):= \frac{1}{2} \mathcal T(\vec g)- \mathcal P(\vec g).
\]
\\

Similarly, in $d=6,$ we can prove the result in Theorem \ref{theo-blow-6d} provided that we replace \eqref{cond-blow-6d} with 
\begin{equation}\tag{BC$^\prime_{6d}$}
\mathcal E(\vec u_0)<\mathcal E_0(\vec\varphi) \quad \& \quad \mathcal T(\vec u_0)>\mathcal T(\vec \varphi),
\end{equation}
where $\vec\varphi=(\varphi_1,\dots, \varphi_N)$ is the ground state related to
\begin{align} \label{ell-equ-vec-6d}
-b_j\Delta \varphi_j=f_j(\vec\varphi), \quad j\in\{1,\dots, N\}.
\end{align}
Here $\vec \varphi$ is a ground state related to \eqref{ell-equ-vec-6d} if it is a non-trivial real valued solution to \eqref{ell-equ-vec-6d} that minimizes the functional $\mathcal E_0$ over all non-trivial real-valued solutions of \eqref{ell-equ-vec-6d}. Note that blow-up results similar to Theorems \ref{theo-blow-5d} and \ref{theo-blow-6d} for radial solutions to \eqref{GQNLS} were established in \cite{NP-blow}. Thus our extensions are for anisotropic solutions. 

\noindent As pointed-out in \cite{NP-blow}, the non-mass-resonance condition $0<\kappa\neq\frac12$ for \eqref{SNLS} in Theorems \ref{theo-blow-5d} and  \ref{theo-blow-6d}, corresponds to the following analogous condition for \eqref{GQNLS}:
\begin{align} \label{cond-non-mass-reso-GQNLS}
\ima  \sum_{j=1}^N\frac{a_j}{2b_j}f_j(\boldsymbol{z})\bar z_j  \ne 0, \quad \forall \boldsymbol{z}=(z_1,\dots,z_N)\in\C^N.
\end{align}

In the mass-critical case $d=4$, we have the following blow-up results for \eqref{GQNLS}. 

\begin{theorem} \label{theo-blow-grow-4d-GQNLS}
	Let $d=4$ and assume that \eqref{cond-non-mass-reso-GQNLS} holds. Let $\vec u_0=(u_{0,1}, \dots, u_{0,N}) \in H^1 \times \dots \times H^1$ be radially symmetric satisfying $\mathcal E(\vec u_0) <0$. Then the corresponding solution to \eqref{GQNLS} either blows-up forward in finite time, i.e. $T^*<\infty$, or it blows-up in infinite time in the sense that $T^*=\infty$ and
	\[
	\mathcal T (\vec u(t)) \geq Ct^2
	\]
	for all $t\geq t_0$, where $C>0$ and $t_0\gg 1$ depend only on $\mathcal M(\vec u_0)$, and $\mathcal E(\vec u_0)$. Moreover, if we assume
	\begin{align} \label{cond-mass-reso-GQNLS}
	\ima  \sum_{j=1}^N\frac{a_j}{2b_j}f_j(\boldsymbol{z})\bar z_j  = 0, \quad \forall \boldsymbol{z}=(z_1,\dots,z_N)\in\C^N
	\end{align} 
	instead of \eqref{cond-non-mass-reso-GQNLS}, then the corresponding solution to \eqref{GQNLS} blows-up in finite time.
\end{theorem}
The proof of this result follows from a similar argument as that for Theorem \ref{theo-blow-grow-4d} using a refined localized virial estimates (see Lemma \ref{lem-viri-est-rad-4d-GQNLS}). Our result is new even under the mass-resonance condition \eqref{cond-mass-reso-GQNLS}. Note that the finite time blow-up for \eqref{GQNLS} in the mass-critical case $d=4$ was proved in \cite[Theorem 5.11]{NP-CCM} only for finite variance solutions.

\begin{theorem} \label{theo-blow-grow-4d-Sigma-SQNLS}
	Let $d=4$ and assume that \eqref{cond-non-mass-reso-GQNLS} holds. Let $\vec u_0 = (u_{0,1}, \dots, u_{0,N}) \in \Sigma_4 \times \dots \times \Sigma_4$ satisfy $\mathcal E(\vec u_0)<0$. Then the corresponding solution to \eqref{GQNLS} either blows-up forward in finite time, i.e. $T^*<\infty$, or $T^*=\infty$ and there exists a time sequence $t_n \rightarrow \infty$ such that $\|(u_1(t_n), \dots, u_N(t_n))\|_{H^1 \times \dots \times H^1} \rightarrow \infty$ as $n\rightarrow \infty$. If we assume \eqref{cond-mass-reso-GQNLS} instead of \eqref{cond-non-mass-reso-GQNLS}, then either $T^* <\infty$ or $T^*=\infty$  and  $\|(\partial_4 u_1(t_n), \dots, \partial_4 u_N(t_n))\|_{L^2\times \dots \times L^2} \rightarrow \infty$ for some diverging time sequence $t_n\rightarrow \infty.$ 
\end{theorem}

 Similarly to Theorem \ref{theo-blow-grow-Sigma-4d}, the proof of Theorem \ref{theo-blow-grow-4d-Sigma-SQNLS} is based on refined localized virial estimates for anisotropic solutions to \eqref{GQNLS} (see Lemma \ref{lem-viri-est-cyli-4d-GQNLS}). Therefore, we will omit the details of the proof.\\

The paper is organized as follows. In Section \ref{S3}, we recall some useful properties of ground states related to \eqref{ground-1} and \eqref{ground-2}. We also prove some variational estimates associated to blow-up conditions given in Theorems \ref{theo-blow-5d} and \ref{theo-blow-6d}. Section \ref{S4} is devoted to various localized virial estimates for radial and anisotropic solutions to \eqref{SNLS}. The proofs of our main results are  given in Section \ref{S5}. Finally, we prove in Appendix \ref{S6} some localized virial estimates for the general system  \eqref{GQNLS} of NLS with quadratic interactions.

 	\section{Variational analysis}
 	\label{S3}
 	\setcounter{equation}{0}
In this section, we report some useful properties of ground states related to \eqref{ground-1} and \eqref{ground-2}.  Then we use them to get some a-priori uniform-in-time estimates for the Pohozaev functional evaluated at the solutions to the corresponding time-dependent equations. 
	\subsection{Variational inequalities}
	We first recall the following Gagliardo-Nirenberg type inequalities due to \cite{HOT} (see also \cite{NP-DPDE}): for $1\leq d\leq 5$,
	\begin{align} \label{GN-ineq}
	P(f,g) \leq C_{\GN} [M(f,g)]^{\frac{6-d}{4}} [T(f,g)]^{\frac{d}{4}}, \quad (f,g) \in H^1 \times H^1.
	\end{align}
	The optimal constant in \eqref{GN-ineq} is attained by any ground state $(\phi,\psi)$ related to \eqref{ground-1}, i.e.
	\[
	C_{\GN} =\frac{P(\phi,\psi)}{[M(\phi,\psi)]^{\frac{6-d}{4}} [T(\phi,\psi)]^{\frac{d}{4}}}.
	\]
	This result was first shown by Hayashi, Ozawa, and Tanaka \cite[Theorem 5.1]{HOT} (for $d=4$), and recently by Noguera and Pastor \cite[Corollary 2.10]{NP-DPDE} (for $1\leq d\leq 5$). We also have the following Pohozaev's identity:
	\begin{align} \label{poho-iden}
	M(\phi,\psi) = \frac{6-d}{d} T(\phi,\psi) = \frac{6-d}{2} P(\phi,\psi).
	\end{align}
	It follows that
	\[
	C_{\GN} = \frac{2}{d [M(\phi,\psi)]^{\frac{6-d}{4}} [T(\phi,\psi)]^{\frac{d-4}{4}}}.
	\]
	When $d=4$, we have
	\begin{align} \label{opti-cons-4d}
	C_{\GN} = \frac{1}{2} [M(\phi,\psi)]^{-\frac{1}{2}}.
	\end{align}
	Although the uniqueness (up to symmetries) of ground states related to \eqref{ground-1} is not known yet, \eqref{opti-cons-4d} shows that the mass of ground states does not depend on the choice of a ground state $(\phi,\psi)$. \\
	
	In the case $d=5$, we have
	\begin{align} \label{opti-cons-5d}
	C_{\GN} = \frac{2}{5} [M(\phi,\psi) T(\phi,\psi)]^{-\frac{1}{4}}
	\end{align}
	and
	\begin{align} \label{E-M-5d}
	E(\phi,\psi) = \frac{1}{10}  T(\phi,\psi) = \frac{1}{4} P(\phi,\psi).
	\end{align}
	In particular, the quantities
	\begin{align} \label{inva-quan-5d}
	E(\phi,\psi) M(\phi,\psi), \quad T(\phi,\psi) M(\phi,\psi), \quad P(\phi,\psi) M(\phi,\psi)
	\end{align}
	do not depend on the choice of a ground state $(\phi,\psi)$.\\
	
	When $d=6$, we have the following Sobolev type inequality:
	\begin{align} \label{Sobo-ineq}
	P(f,g) \leq C_{\Sob} [T(f,g)]^{\frac{3}{2}}, \quad (f,g) \in \dot{H}^1 \times \dot{H}^1.
	\end{align}
	It was shown in \cite[Theorem 3.3]{NP-blow} that the sharp constant in \eqref{Sobo-ineq} is achieved by a ground state $(\phi,\psi)$ related to \eqref{ground-2}, i.e.
	\[
	C_{\Sob} = \frac{P(\phi,\psi)}{ [T(\phi,\psi)]^{\frac{3}{2}}}.
	\]
	Using the following identity
	\[
	T(\phi,\psi) = 3 P(\phi,\psi),
	\]
	we see that
	\begin{align} \label{shap-cons-6d}
	C_{\Sob} = \frac{1}{3} [T(\phi,\psi)]^{-\frac{1}{2}}
	\end{align}
	and
	\begin{align} \label{E-K-6d}
	E(\phi,\psi) = \frac{1}{6} T(\phi,\psi).
	\end{align}
	This shows in particular that $E(\phi,\psi)$ and $T(\phi,\psi)$ do not depend on the choice of a ground state $(\phi,\psi)$.
	
	\subsection{Variational estimates}
	In this section, we characterize the blow-up region defined in \eqref{cond-blow} (see \eqref{cond-blow-6d} for the energy critical case) in terms of the sign of the Pohozaev functional $G$ defined in \eqref{defi-G}. For similar analysis in the context of the classical NLS equation, we refer to our previous works \cite{BF19, DFH}.
	
	\begin{lemma} \label{lem-equi-5d}
	Let $d=5$, $\kappa>0$, and $(\phi,\psi)$ be a ground state related \eqref{ground-1}. Denote
	\begin{align} \label{A}
	\Ac:= \left\{ (f,g) \in H^1\times H^1\ \hbox{ s. t. }
	\begin{array}{rcl}
	E(f,g) M(f,g) &<& E(\phi,\psi) M(\phi,\psi) \\
	T(f,g) M(f,g) &>& T(\phi,\psi) M(\phi,\psi)
	\end{array}
	\right\}
	\end{align}
	and
	\begin{align} \label{tilde-A}
	\tilde{\Ac}:= \left\{ 
	(f,g) \in H^1\times H^1 \ \hbox{ s. t. } 
	\begin{array}{rcl}
	E(f,g) M(f,g) &<& E(\phi,\psi) M(\phi,\psi) \\
	G(f,g) &<& 0
	\end{array}
	\right\}.
	\end{align}
	Then $\Ac \equiv \tilde{\Ac}$.	
	\end{lemma}
	
	\begin{proof}
		Let $(f,g) \in \Ac$. We will show that $G(f,g) <0$, hence $(f,g) \in \tilde{\Ac}$. We have
		\begin{align*}
		G(f,g) M(f,g) = \left(\frac{5}{2} E(f,g) - \frac{1}{4} T(f,g) \right) M(f,g) <\frac{5}{2} E(\phi,\psi) M(\phi,\psi) - \frac{1}{4} T(\phi,\psi) M(\phi,\psi),
		\end{align*}
		hence $G(f,g)<0$ by using \eqref{E-M-5d}.\\
		
\noindent 		Now let $(f,g) \in \tilde{\Ac}$. We will show that $T(f,g) M(f,g) > T(\phi,\psi) M(\phi,\psi)$, so $(f,g) \in \Ac$. Indeed, as $G(f,g) <0$, we use  \eqref{GN-ineq} to have
		\begin{align*}
		T(f,g) <\frac{5}{2} P(f,g) \leq \frac{5}{2} C_{\GN} [M(f,g)]^{\frac{1}{4}} [T(f,g)]^{\frac{5}{4}}.
		\end{align*}
		In particular, we have
		\[
		\left(T(f,g) M(f,g)\right)^{\frac{1}{4}} >\frac{2}{5} C_{\GN}^{-1} 
		\]
		which, by \eqref{opti-cons-5d}, implies that
		\[
		T(f,g) M(f,g) > T(\phi,\psi) M(\phi,\psi).
		\]
		The proof is complete.
	\end{proof}
	
	\begin{lemma} \label{lem-uppe-boun-5d}
		Let $d=5$, $\kappa >0$, and $(\phi,\psi)$ be a ground state related to \eqref{ground-1}. Let $(u_0,v_0) \in H^1 \times H^1$ satisfy \eqref{cond-blow}. Let $(u,v)$ be the corresponding solution to \eqref{SNLS} defined on the maximal forward time interval $[0,T^*)$. Then there exist positive constants $\vareps$ and $c$ such that
		\begin{align} \label{uppe-boun-5d}
		G(u(t),v(t)) + \vareps T(u(t),v(t)) \leq -c 
		\end{align}
		for all $t\in [0,T^*)$.
	\end{lemma}
	\begin{proof}
In case of negative energy the proof is straightforward: indeed, if $E(u_0,v_0)<0$, the conservation of energy yields
		\[
		G(u(t),v(t)) = \frac{5}{2} E(u(t),v(t)) -\frac{1}{4} T(u(t),v(t)) = \frac{5}{2} E(u_0,v_0) - \frac{1}{4} T(u(t),v(t)).
		\]
		This shows \eqref{uppe-boun-5d} with $\vareps = \frac{1}{4}$ and $c= -\frac{5}{2} E(u_0,v_0)>0$.\\
		
\noindent 	Let us now focus on the case $E(u_0,v_0)\geq 0$. By \eqref{GN-ineq}, we have for all $t\in [0,T^*)$,
		\begin{equation}\label{eq:F}
		\begin{aligned}
		E(u(t),v(t)) M(u(t),v(t)) &= \frac{1}{2} T(u(t),v(t)) M(u(t),v(t)) - P(u(t),v(t)) M(u(t),v(t)) \\
		&\geq \frac{1}{2} T(u(t),v(t)) M(u(t),v(t)) - C_{\GN} \left(T(u(t),v(t)) M(u(t),v(t)) \right)^{\frac{5}{4}} \\
		&= F\left( T(u(t),v(t)) M(u(t),v(t))\right),
		\end{aligned}
		\end{equation}
		where
		\[
		F(\lambda):= \frac{1}{2} \lambda - C_{\GN} \lambda^{\frac{5}{4}}.
		\]
		Using \eqref{opti-cons-5d} and \eqref{E-M-5d}, we see that
		\begin{align*}
		F\left( T(\phi,\psi) M(\phi,\psi) \right) &= \frac{1}{2} T(\phi,\psi) M(\phi,\psi) - C_{\GN} \left(T(\phi,\psi) M(\phi,\psi) \right)^{\frac{5}{4}} \\
		&=\frac{1}{10} T(\phi,\psi) M(\phi,\psi) = E(\phi,\psi) M(\phi,\psi).
		\end{align*}
		Thanks to the first condition in \eqref{cond-blow} and the conservation laws of mass and energy, we can continue the estimate \eqref{eq:F} as
		\begin{equation*}
		\begin{aligned}
		F\left( T(u(t),v(t)) M(u(t),v(t))\right)& \leq E(u_0,v_0) M(u_0,v_0) \\
		&< E(\phi,\psi)M(\phi,\psi)  = F \left( T(\phi,\psi) M(\phi,\psi) \right), \quad \forall\,t\in [0,T^*).
		\end{aligned}
		\end{equation*}
		By the continuity argument and the second condition in \eqref{cond-blow}, we infer that 
		\begin{align} \label{est-solu-5d}
		T(u(t),v(t)) M(u(t),v(t)) > T(\phi,\psi) M(\phi,\psi)
		\end{align}
		for all $t\in [0,T^*)$. Next we use the first condition in \eqref{cond-blow} to pick $\rho:= \rho(u_0,v_0, \phi,\psi)>0$ so that
		\begin{align} \label{defi-rho-5d}
		E(u_0,v_0) M(u_0,v_0) \leq (1-\rho) E(\phi,\psi) M(\phi,\psi).
		\end{align}
		It follows that
		\[
		F\left( T(u(t),v(t)) M(u(t),v(t)) \right) \leq (1-\rho) E(\phi,\psi) M(\phi, \psi).
		\]
		Using the fact that
		\[
		E(\phi,\psi) M(\phi,\psi) = \frac{1}{10} T(\phi,\psi) M(\phi,\psi) = \frac{1}{4}C_{\GN} \left( T(\phi,\psi) M(\phi,\psi)\right)^{\frac{5}{4}},
		\]
		we infer that
		\begin{align} \label{est-rho-5d}
		5 \frac{T(u(t), v(t)) M(u(t),v(t))}{T(\phi,\psi) M(\phi,\psi)} - 4 \left(\frac{T(u(t), v(t)) M(u(t),v(t))}{T(\phi,\psi) M(\phi,\psi)} \right)^{\frac{5}{4}} \leq 1-\rho
		\end{align}
		for all $t\in [0,T^*)$. We consider $g(\lambda):= 5 \lambda -4 \lambda^{\frac{5}{4}}$ for $\lambda>1$. Note that the condition $\lambda>1$ is due to \eqref{est-solu-5d}. We see that $g^\prime(1)=0$ and $g$ is strictly decreasing on $(1,\infty)$. It follows from \eqref{est-rho-5d} that there exists $\nu:= \nu(\rho)>0$ such that $\lambda>1+\nu$. In particular, we have
		\begin{align} \label{refi-est-solu-5d}
		T(u(t),v(t)) M(u(t),v(t)) \geq (1+\nu) T(\phi,\psi) M(\phi,\psi)
		\end{align}
		for all $t\in [0,T^*)$. Now let $\vareps>0$ to be chosen later. By the conservation of mass and energy, \eqref{E-M-5d}, \eqref{defi-rho-5d}, and \eqref{refi-est-solu-5d}, we have for all $t\in [0,T^*)$,
		\begin{align*}
		\Big( G(u(t),v(t)) &+ \vareps T(u(t),v(t)) \Big) M(u(t),v(t)) \\
		&= \frac{5}{2} E(u(t),v(t)) M(u(t),v(t)) - \left(\frac{1}{4}-\vareps\right) T(u(t),v(t)) M(u(t),v(t)) \\
		&=\frac{5}{2} E(u_0,v_0) M(u_0,v_0) -  \left(\frac{1}{4}-\vareps\right) T(u(t),v(t)) M(u(t),v(t)) \\
		&\leq \frac{5}{2}(1-\rho) E(\phi,\psi) M(\phi,\psi) - \left(\frac{1}{4}-\vareps\right) (1+\nu) T(\phi,\psi) M(\phi,\psi) \\
		&= \left( -\frac{1}{4}(\rho+\nu) +\vareps (1+\nu)\right) T(\phi,\psi) M(\phi,\psi).
		\end{align*} 
		By taking $0<\vareps <\frac{\rho+\nu}{4(1+\nu)}$  and using the conservation of mass, we have \eqref{uppe-boun-5d} with 
		\[
		c= \left(\frac{1}{4}(\rho+\nu) - \vareps(1+\nu)\right) T(\phi,\psi) \frac{M(\phi,\psi)}{M(u_0,v_0)}>0.
		\]
		The proof is complete.
	\end{proof}

	\begin{lemma} \label{lem-equi-6d}
		Let $d=6$, $\kappa>0$, and $(\phi,\psi)$ be a ground state related \eqref{ground-2}. Denote
		\begin{align} \label{B}
		\Bc:= \left\{ (f,g) \in H^1\times H^1 \ \hbox{ s. t. } 
		\begin{array}{rcl}
		E(f,g)  &<& E(\phi,\psi) \\
		T(f,g) &>& T(\phi,\psi)
		\end{array}
		\right\}
		\end{align}
		and
		\begin{align} \label{tilde-B}
		\tilde{\Bc}:= \left\{ 
		(f,g) \in H^1\times H^1 \ \hbox{ s. t. } 
		\begin{array}{rcl}
		E(f,g) &<& E(\phi,\psi) \\
		G(f,g) &<& 0
		\end{array}
		\right\}.
		\end{align}
		Then $\Bc \equiv \tilde{\Bc}$.	
	\end{lemma}
	
	\begin{proof}
		Let $(f,g) \in \Bc$. We will show that $G(f,g) <0$. Indeed, by \eqref{E-K-6d}, we have
		\[
		G(f,g)= 3 E(f,g) - \frac{1}{2}T(f,g) < 3E(\phi,\psi) - \frac{1}{2}T(\phi,\psi) =0.
		\]
		
	\noindent 	Let us consider now $(f,g) \in \tilde{\Bc}$. As $G(f,g) <0$, we have from \eqref{Sobo-ineq} that
		\[
		T(f,g) < 3 P(f,g) \leq 3 C_{\Sob} [T(f,g)]^{\frac{3}{2}} 
		\]
		or equivalently $[T(f,g)]^{\frac{1}{2}} > \frac{1}{3} C_{\Sob}^{-1}$. This shows that $T(f,g)> T(\phi,\psi)$ thanks to \eqref{shap-cons-6d}.
	\end{proof}

	\begin{lemma} \label{lem-uppe-boun-6d}
		Let $d=6$, $\kappa>0$, and $(\phi,\psi)$ be a ground state related \eqref{ground-2}. Let $(u_0,v_0) \in H^1 \times H^1$ satisfy  \eqref{cond-blow-6d}. Let $(u,v)$ be the corresponding solution to \eqref{SNLS} defined on the maximal forward time interval $[0,T^*)$. Then there exist positive constants $\vareps$ and $c$ such that
		\begin{align} \label{uppe-boun-6d}
		G(u(t),v(t)) + \vareps T(u(t),v(t)) \leq -c 
		\end{align}
		for all $t\in [0,T^*)$.
	\end{lemma}
	
	\begin{proof}
		The proof is similar to that of Lemma \ref{lem-uppe-boun-5d} using \eqref{shap-cons-6d} and \eqref{E-K-6d}. We thus omit the details.
	\end{proof}

	\section{Localized virial estimates}
	\label{S4}
	\setcounter{equation}{0}
	In this section we prove the preliminary and fundamental estimates we need for the proof of our main Theorems. We start with the following virial identity (see e.g. \cite[(4.34)]{WY}).
	\begin{lemma} \label{lem-viri}
		Let $d\geq 1$ and $\kappa>0$. Let $\varphi: \R^d \rightarrow \R$ be a sufficiently smooth and decaying function. Let $(u,v)$ be a $H^1$-solution to \eqref{SNLS} defined on the maximal forward time interval $[0,T^*)$. Define
	 	\begin{align} \label{defi-M-varphi}
	 	M_\varphi(t):= 2\ima \int \nabla \varphi(x) \cdot \left( \nabla u(t,x) \overline{u}(t,x) + \nabla v(t,x) \overline{v}(t,x)\right) dx.
	 	\end{align}
	 	Then we have for all $t\in [0,T^*)$,
	 	\begin{align*}
	 	\frac{d}{dt} M_\varphi(t) &= - \int \Delta^2 \varphi(x) \left(|u(t,x)|^2 + \kappa |v(t,x)|^2 \right) dx \\
	 	&\mathrel{\phantom{=}}+ 4\sum_{j,k=1}^d \rea \int \partial^2_{jk} \varphi(x) \left( \partial_j u(t,x) \partial_k \overline{u}(t,x) + \kappa \partial_j v(t,x) \partial_k \overline{v}(t,x) \right) dx \\
	 	&\mathrel{\phantom{=}} - 2 \rea \int \Delta \varphi(x)  v(t,x) \overline{u}^2(t,x) dx.
	 	\end{align*}
	 \end{lemma}
	 
	 \noindent The above identity can be checked by formal computations. The rigorous proof can be done by performing a standard approximation trick (see e.g. \cite[Section 6.5]{Cazenave}).
	 
	 \begin{remark} From now on we denote $r=|x|.$ \label{rem-viri} ~
	 	\begin{itemize}[leftmargin=6mm]
	 		\item[(1)] If $\varphi(x) = |x|^2$, then
	 		\[
	 		\frac{d}{dt} M_{|x|^2}(t) = 8 G(u(t),v(t)),
	 		\]
	 		where $G$ is as in \eqref{defi-G}.
	 		\item[(2)] If $\varphi$ is radially symmetric, then using the fact that
	 		\[
	 		\partial_j = \frac{x_j}{r} \partial_r, \quad \partial^2_{jk} = \left( \frac{\delta_{jk}}{r} - \frac{x_jx_k}{r^3} \right) \partial_r + \frac{x_j x_k}{r^2} \partial^2_r,
	 		\]
	 		we have
	 		\begin{align*}
	 		\sum_{j,k=1}^d \rea &\int \partial^2_{jk} \varphi(x) \partial_j u (t,x) \partial_k \overline{u}(t,x) dx \\
	 		&= \int \frac{\varphi'(r)}{r} |\nabla u(t,x)|^2 dx + \int \left(\frac{\varphi''(r)}{r^2}-\frac{\varphi'(r)}{r^3}\right) |x \cdot \nabla u(t,x)|^2 dx, 
	 		\end{align*}
	 		where $r=|x|$. In particular, we have
	 		\begin{align*}
	 		\frac{d}{dt} M_\varphi(t) &= -\int \Delta^2 \varphi(x) \left(|u(t,x)|^2 + \kappa |v(t,x)|^2 \right) dx \\
	 		& + 4\int \frac{\varphi'(r)}{r} \left(|\nabla u(t,x)|^2 + \kappa |\nabla v(t,x)|^2 \right) dx \\
	 		& + 4 \int \left(\frac{\varphi''(r)}{r^2} - \frac{\varphi'(r)}{r^3} \right) \left(|x\cdot \nabla u(t,x)|^2 + \kappa |x\cdot \nabla u(t,x)|^2 \right) dx \\
	 		& - 2 \rea \int \Delta \varphi(x) v(t,x) \overline{u}^2(t,x) dx.
	 		\end{align*}
	 		\item[(3)] If $\varphi$ is radial and $(u,v)$ is also radial, then 
	 		\begin{align*}
	 		\frac{d}{dt} M_\varphi(t) &= -\int \Delta^2 \varphi(x) \left(|u(t,x)|^2 + \kappa |v(t,x)|^2\right) dx \\
	 		& + 4 \int \varphi''(r) \left(|\nabla u(t,x)|^2 + \kappa |\nabla v(t,x)|^2\right) dx \\
	 		& - 2 \rea \int \Delta \varphi(x) v(t,x) \overline{u}^2(t,x) dx.
	 		\end{align*}
	 		\item[(4)] Let $d\geq 3$ and denote $x=(y,x_d)$ with $y=(x_1, \dots, x_{d-1}) \in \R^{d-1}$ and $x_d \in \R$. Let $\psi: \R^{d-1} \rightarrow \R$ be a sufficiently smooth and decaying function. Set $\varphi(x) = \psi(y) + x_d^2$. We have
	 		\begin{align*}
	 		\frac{d}{dt}M_\varphi(t) &= -\int \Delta^2_y \psi(y) \left(|u(t,x)|^2 + \kappa |v(t,x)|^2\right) dx \\
	 		& + 4\sum_{j,k=1}^{d-1} \rea \int \partial^2_{jk} \psi(y) \left(\partial_j u(t,x) \partial_k \overline{u}(t,x) + \kappa \partial_j v(t,x) \partial_k \overline{v}(t,x)\right) dx \\
	 		& - 2 \rea \int \Delta_y \psi(y) v(t,x) \overline{u}^2(t,x)dx \\
	 		& + 8 \left(\|\partial_d u(t)\|^2_{L^2} + \kappa \|\partial_d v(t)\|^2_{L^2}\right) - 4 P(u(t),v(t)).
	 		\end{align*}
	 		Moreover, if $(u(t),v(t)) \in \Sigma_d \times \Sigma_d$ for all $t\in [0,T^*)$, then we have
	 		\begin{align*}
	 		\frac{d}{dt} M_\varphi(t) &= -\int \Delta^2_y \psi(y) \left(|u(t,x)|^2 + \kappa |v(t,x)|^2\right) dx \\
	 		& + 4\int \psi''(\rho) \left(|\nabla_y u(t,x)|^2 + \kappa |\nabla_y v(t,x)|^2\right) dx \\
	 		& - 2 \rea \int \Delta_y \psi(y)v(t,x) \overline{u}^2(t,x) dx \\
	 		& + 8 \left(\|\partial_d u(t)\|^2_{L^2} + \kappa \|\partial_d v(t)\|^2_{L^2}\right) - 4 P(u(t),v(t)),
	 		\end{align*}
	 		where $\rho = |y|$.
	 	\end{itemize}
	 \end{remark}
	 
	 Let $\chi: [0,\infty) \rightarrow [0,\infty)$ be a sufficiently smooth function satisfying
	 \begin{align} \label{defi-chi}
	 \chi(s):= \left\{
	 \begin{array}{ccl}
	 s^2 &\text{if}& 0\leq s \leq 1, \\
	 \text{const.} &\text{if}& s\geq 2, 
	 \end{array}
	 \right.
	 \quad \chi''(s) \leq 2, \quad \forall s \geq 0.
	 \end{align}
	 Given $R>1$, we define, by rescaling, the radial function $\varphi_R: \R^d \rightarrow \R$ by
	 \begin{align} \label{defi-varphi-R}
	 \varphi_R(x) = \varphi_R(r) := R^2 \chi(r/R).
	 \end{align}
	 
In the mass-critical case, we have the following refined (with respect to the one in \cite{IKN-NA}) radial localized virial estimate.
	 
	 \begin{lemma}\label{lem-loca-viri-est-rad-2} 
	 	Let $d=4$ and $\kappa >0$. Let $(u,v)$ be a radial $H^1$-solution to \eqref{SNLS} defined on the maximal forward time interval $[0,T^*)$. Let $\varphi_R$ be as in \eqref{defi-varphi-R} and denote $M_{\varphi_R}(t)$ as in \eqref{defi-M-varphi}. Then we have for all $t\in [0,T^*)$,
	 	\begin{align} \label{loca-viri-est-rad-2}
	 	\frac{d}{dt} M_{\varphi_R}(t) &\leq 16 E(u(t),v(t)) -4 \int \left(\theta_{1,R}(r)- CR^{-\frac{3}{2}} (\theta_{2,R}(x))^2\right) |\nabla u(t,x)|^2 dx  + o_R(1)
	 	\end{align}
	 	for some constant $C>0$ depending only on $\kappa$ and $M(u_0,v_0)$, where
	 	\begin{align} \label{defi-theta-12-R}
	 	\theta_{1,R}(r) := 2-\varphi''_R(r), \quad \theta_{2,R}(x) = 8-\Delta \varphi_R(x).
	 	\end{align}
	 \end{lemma}
 	
 	\begin{proof}
 		 By Item (3) of Remark \ref{rem-viri}, we have for all $t\in [0,T^*)$,
\begin{equation}\label{vir:4d-rev}
	 	\begin{aligned}
	 	\frac{d}{dt} M_{\varphi_R}(t) &= -\int \Delta^2 \varphi_R(x) \left(|u(t,x)|^2 + \kappa |v(t,x)|^2 \right) dx + 4\int \varphi''_R(r) \left( |\nabla u(t,x)|^2 + \kappa |\nabla v(t,x)|^2 \right) dx \\
	 	& - 2\rea \int \Delta \varphi_R v(t,x) \overline{u}^2 (t,x) dx\\
	&= 8 G(u(t),v(t)) - 4 \int (2-\varphi''_R(r)) \left( |\nabla u(t,x)|^2 + \kappa |\nabla v(t,x)|^2\right) dx \\
	 	& - \int \Delta^2 \varphi_R(x) \left(|u(t,x)|^2 + \kappa |v(t,x)|^2\right) dx + 2 \rea \int (2d-\Delta \varphi_R(x)) v(t,x) \overline{u}^2(t,x) dx.
	 	\end{aligned}
		\end{equation}
By the fact that  $\|\Delta^2 \varphi_R\|_{L^\infty} \lesssim R^{-2}$ together with the conservation of mass, we get the decay 
	 	\[
	 	\left| \int \Delta^2 \varphi_R(x) \left(|u(t,x)|^2 +\kappa |v(t,x)|^2\right) dx \right| \lesssim R^{-2}.
	 	\]
Furthermore, by using that $\varphi''_R (r) \leq 2,$ and by noting that   $G(u(t),v(t))=2 E(u(t),v(t))$ if $d=4,$ \eqref{vir:4d-rev} can be controlled by

 		\begin{align*}
 		\frac{d}{dt}M_{\varphi_R}(t) \leq 16E(u(t),v(t)) - 4 \int \theta_{1,R}(r) |\nabla u(t,x)|^2 dx + 2 \rea \int \theta_{2,R}(x) v(t,x) \overline{u}^2(t,x) dx + CR^{-2}.
 		\end{align*}
 		We estimate
 		\begin{align*}
 		\left| \rea \int \theta_{2,R}(x) v(t,x) \overline{u}^2(t,x) dx \right| &\leq \sup_{|x|\geq R} |\theta_{2,R}(x) u(t,x)| \|v(t)\|_{L^2} \|u(t)\|_{L^2} \\
 		&\lesssim R^{-\frac{3}{2}} \|\nabla (\theta_{2,R} u(t))\|_{L^2}^{\frac{1}{2}} \|\theta_{2,R} u(t)\|^{\frac{1}{2}}_{L^2} \|v(t)\|_{L^2} \|u(t)\|_{L^2} \\
 		&\lesssim R^{-\frac{3}{2}} \|\nabla (\theta_{2,R} u(t))\|^{\frac{1}{2}}_{L^2},
 		\end{align*}
 		where we have used the conservation of mass in the last estimate. Note that $\theta_{2,R}(x)=0$ for $|x|\leq R$. As $\|\nabla \theta_{2,R}\|_{L^\infty} \lesssim 1$, the conservation of mass implies that
 		\[
 		\|\nabla(\theta_{2,R} u(t))\|_{L^2} \lesssim \|\nabla \theta_{2,R}\|_{L^\infty} \|u(t)\|_{L^2} + \|\theta_{2,R} \nabla u(t)\|_{L^2} \lesssim \|\theta_{2,R} \nabla u(t)\|_{L^2} + 1.
 		\]
 		It follows that
 		\begin{align*}
 		\left| \rea \int \theta_{2,R}(x) v(t,x) \overline{u}^2(t,x) dx \right| & \lesssim R^{-\frac{3}{2}} \left(\|\theta_{2,R} \nabla u(t)\|_{L^2} + 1 \right)^{\frac{1}{2}} \\
 		&\lesssim R^{-\frac{3}{2}} \left( \|\theta_{2,R} \nabla u(t)\|^2_{L^2} +1\right).
 		\end{align*}
 		Therefore, we obtain
 		\[
 		\frac{d}{dt}M_{\varphi_R}(t) \leq 16 E(u(t),v(t)) - 4 \int \left(\theta_{1,R}(r) - CR^{-\frac{3}{2}} (\theta_{2,R}(x))^2\right) |\nabla u(t,x)|^2 dx + CR^{-2} +CR^{-\frac{3}{2}}.
 		\]
 		The proof is complete.
 	\end{proof}
 	
 	Next we derive localized virial estimates for cylindrically symmetric solutions. To this end, we introduce
 	\begin{align} \label{defi-psi-R}
 	\psi_R(y)= \psi_R(\rho) := R^2 \chi(\rho/R), \quad \rho =|y|
 	\end{align}
 	and set
 	\begin{align} \label{defi-varphi-R-psi}
 	\varphi_R(x):= \psi_R(y) + x_d^2.
 	\end{align}
 	
 	\begin{lemma}[Cylindrical localized virial estimate I] \label{lem-loca-viri-est-cyli-1}
 		Let $d=5,6$, and $\kappa>0$. Let $(u,v)$ be a $\Sigma_d$-solution to \eqref{SNLS} defined on the maximal forward time interval $[0,T^*)$. Let $\varphi_R$ be as in \eqref{defi-varphi-R-psi} and denote $M_{\varphi_R}(t)$ as in \eqref{defi-M-varphi}. Then we have for all $t\in [0,T^*)$,
 		\begin{align} \label{loca-viri-est-cyli-1}
 		\frac{d}{dt} M_{\varphi_R}(t) \leq 8 G(u(t),v(t)) + CR^{-\frac{d-2}{2}} \|\nabla u(t)\|^2_{L^2} + o_R(1)
 		\end{align}
 		for some constant $C>0$ depending only on $d,\kappa$, and $M(u_0,v_0)$.
 	\end{lemma}
 	
 	\begin{proof}
 		By Item (4) of Remark \ref{rem-viri}, we have for all $t\in [0,T^*)$,
 		\begin{align*}
 		\frac{d}{dt} M_{\varphi_R}(t) &=  -\int \Delta^2_y \psi_R(y) \left(|u(t,x)|^2 + \kappa |v(t,x)|^2\right) dx \\
 		& + 4\int \psi''_R(\rho) \left(|\nabla_y u(t,x)|^2 + \kappa |\nabla_y v(t,x)|^2\right) dx \\
 		& - 2 \rea \int \Delta_y \psi_R(y)v(t,x) \overline{u}^2(t,x) dx \\
 		& + 8 \left(\|\partial_d u(t)\|^2_{L^2} + \kappa \|\partial_d v(t)\|^2_{L^2}\right) - 4 P(u(t),v(t)).
 		\end{align*}
 		It follows that
 		\begin{align*}
 		\frac{d}{dt} M_{\varphi_R}(t) &\leq 8 G(u(t),v(t)) +CR^{-2} - 4 \int (2-\psi''_R(\rho)) \left( |\nabla_y u(t,x)|^2 + \kappa |\nabla_y v(t,x)|^2\right) dx \\
 		& +2 \rea \int (2(d-1)-\Delta_y\psi_R(y)) v(t,x) \overline{u}^2(t,x) dx.
 		\end{align*}
 		As $\psi''_R(\rho) \leq 2$ and $\|\Delta_y \psi_R\|_{L^\infty_x} \lesssim 1$, we have
 		\begin{align} \label{est-cyli}
 		\frac{d}{dt} M_{\varphi_R}(t) \leq 8 G(u(t),v(t)) + CR^{-2} + C \int_{|y|\geq R} |v(t,x) \overline{u}^2(t,x)| dx.
 		\end{align}
 		By the conservation of mass, we have
 		\begin{align}
 		\int_{|y|\geq R} |v(t,x) \overline{u}^2(t,x)|dx &\leq \left( \int_{|y|\geq R} |u(t,x)|^4 dx\right)^{1/2} \left( \int_{|y|\geq R} |v(t,x)|^2 dx\right)^{1/2} \nonumber \\
 		&\lesssim \left( \int_{|y|\geq R} |u(t,x)|^4 dx\right)^{1/2}. \label{est-cyli-0}
 		\end{align}
 		Next we estimate
 		\begin{align*}
 		\int_{|y|\geq R} |u(t,x)|^4 dx &\leq \int_{\R} \|u(t,x_d)\|^2_{L^2_y} \|u(t,x_d)\|^2_{L^\infty_y(|y|\geq R)} dx_d \\
 		&\leq \sup_{x_d \in \R} \|u(t,x_d)\|^2_{L^2_y} \left(\int_{\R} \|u(t,x_d)\|^2_{L^\infty_y(|y|\geq R)} dx_d \right).
 		\end{align*}
 		Set $g(x_d):= \|u(t,x_d)\|^2_{L^2_y}$, we have
 		\begin{align*}
 		g(x_d) = \int_{-\infty}^{x_d} \partial_s g(s) ds &= 2 \int_{-\infty}^{x_d} \rea \int_{\R^{d-1}} \overline{u}(t,y,s) \partial_s u(t,y,s) dy  ds \\
 		&\leq 2 \|u(t)\|_{L^2_x} \|\partial_d u(t)\|_{L^2_x}
 		\end{align*}
 		which, by the conservation of mass, implies that
 		\begin{align} \label{est-cyli-1}
 		\sup_{x_d \in \R} \|u(t,x_d)\|^2_{L^2_y} \lesssim \|\partial_d u(t)\|_{L^2_x}.
 		\end{align}
 		We next use the radial Sobolev embedding, see \cite{CO}, 
		\begin{equation} \label{rad-sobo}
	 	\sup_{x \ne 0} |x|^{\frac{d-1}{2}} |f(x)| \leq C(d) \|\nabla f\|^{\frac{1}{2}}_{L^2} \|f\|^{\frac{1}{2}}_{L^2},
	 	\end{equation} to have
 		\begin{align}
 		\int_{\R} \|u(t,x_d)\|^2_{L^\infty_y(|y|\geq R)} dx_d &\lesssim R^{-d+2}  \int_{\R} \|\nabla_y u(t,x_d)\|_{L^2_y} \|u(t,x_d)\|_{L^2_y} dx_d \nonumber \\
 		&\lesssim R^{-d+2} \left( \int_{\R} \|\nabla_y u(t,x_d)\|^2_{L^2_y} dx_d\right)^{1/2} \left( \int_{\R} \|u(t,x_d)\|^2_{L^2_y} dx_d\right)^{1/2} \nonumber \\
 		&\lesssim R^{-d+2} \|\nabla_y u(t)\|_{L^2_x} \|u(t)\|_{L^2_x} \nonumber \\
 		&\lesssim R^{-d+2} \|\nabla_y u(t)\|_{L^2_x}. \label{est-cyli-2}
 		\end{align}
 		Collecting \eqref{est-cyli-0}, \eqref{est-cyli-1}, and \eqref{est-cyli-2}, we get
 		\begin{align*}
 		\int_{|y|\geq R} |v(t,x) \overline{u}^2(t,x)| dx &\lesssim R^{-\frac{d-2}{2}} \|\nabla_y u(t)\|^{1/2}_{L^2_x} \|\partial_d u(t)\|^{1/2}_{L^2_x} \\
 		&\lesssim R^{-\frac{d-2}{2}} \left(\|\nabla_y u(t)\|_{L^2_x} + \|\partial_d u(t)\|_{L^2_x}\right) \\
 		&\lesssim R^{-\frac{d-2}{2}} \left( \|\nabla u(t)\|^2_{L^2_x} + 1 \right).
 		\end{align*}
 		This together with \eqref{est-cyli} prove \eqref{loca-viri-est-cyli-1}. The proof is complete.
 	\end{proof}
 	
 	We also have the following refined localized virial estimate which will be used in the mass-critical problem.
 	\begin{lemma}[Cylindrical localized virial estimate II] \label{lem-loca-viri-est-cyli-2}
 		Let $d=4$ and $\kappa>0$. Let $(u,v)$ be a $\Sigma_4$-solution to \eqref{SNLS} defined on the maximal forward time interval $[0,T^*)$. Let $\varphi_R$ be as in \eqref{defi-varphi-R-psi} and denote $M_{\varphi_R}(t)$ as in \eqref{defi-M-varphi}. Then we have for all $t\in [0,T^*)$,
 		\begin{align} \label{loca-viri-est-cyli-2}
 		\frac{d}{dt} M_{\varphi_R}(t) &\leq 16 E(u(t),v(t)) -4 \int \left(\vartheta_{1,R}(\rho) - CR^{-1} (\vartheta_{2,R}(y))^2\right) |\nabla_y u(t,x)|^2 dx \nonumber \\
 		& + CR^{-1} \|\partial_4 u(t)\|^2_{L^2} + o_R(1)
 		\end{align}
 		for some constant $C>0$ depending only on $\kappa$ and $M(u_0,v_0)$, where
 		\begin{align} \label{defi-vartheta-12-R}
 		\vartheta_{1,R}(\rho):= 2-\psi''_R(\rho), \quad \vartheta_{2,R}(y):= 6-\Delta_y \psi_R(y).
 		\end{align}
 	\end{lemma}
 	
 	\begin{proof}
 		Estimating as in the proof of Lemma \ref{lem-loca-viri-est-cyli-1}, we have
 		\begin{align*}
 		\frac{d}{dt} M_{\varphi_R}(t) &\leq 16 E(u(t),v(t)) + CR^{-2} - 4 \int \vartheta_{1,R}(\rho) \left( |\nabla_y u(t,x)|^2 + \kappa |\nabla_y v(t,x)|^2\right) dx \\
 		& +2 \rea \int \vartheta_{2,R}(y) v(t,x) \overline{u}^2(t,x) dx.
 		\end{align*}
 		By the conservation of mass, we see that
 		\begin{align*}
 		\left| \rea \int \vartheta_{2,R}(y) v(t,x) \overline{u}^2(t,x) dx \right| &\leq \left( \int (\vartheta_{2,R}(y))^2 |u(t,x)|^4 dx \right)^{1/2} \left(\int |v(t,x)|^2 dx\right)^{1/2} \\
 		&\lesssim \left( \int  (\vartheta_{2,R}(y))^2 |u(t,x)|^4 dx \right)^{1/2}.
 		\end{align*}
 		By the H\"older's inequality, we have
 		\begin{align*}
 		\int  (\vartheta_{2,R}(y))^2 |u(t,x)|^4 dx  &\leq \int_{\R} \|u(t,x_4)\|^2_{L^2_y} \|\vartheta_{2,R} u(t,x_4)\|^2_{L^\infty_y} dx_4 \\
 		&\leq \sup_{x_4 \in \R} \|u(t,x_4)\|^2_{L^2_y} \left( \int_{\R} \|\vartheta_{2,R} u(t,x_4)\|^2_{L^\infty_y} dx_4\right).
 		\end{align*}
 		The first term is treated in \eqref{est-cyli-1}. For the second term, as $\vartheta_{2,R}(y) =0$ for $|y|\leq R$, we use the radial Sobolev embedding \eqref{rad-sobo} to have
 		\begin{align*}
 		\int_{\R} \|\vartheta_{2,R} u(t,x_4)\|^2_{L^\infty_y} dx_4 &\lesssim R^{-2} \int_{\R} \|\nabla_y (\vartheta_{2,R} u(t,x_4))\|_{L^2_y} \|\vartheta_{2,R} u(t,x_4)\|_{L^2_y} dx_4 \\
 		&\lesssim R^{-2} \left( \int_{\R} \|\nabla_y (\vartheta_{2,R} u(t,x_4))\|^2_{L^2_y} dx_4 \right)^{1/2} \left( \int_{\R} \|\vartheta_{2,R} u(t,x_4)\|^2_{L^2_y} dx_4\right)^{1/2} \\
 		&\lesssim R^{-2} \|\nabla_y (\vartheta_{2,R} u(t))\|_{L^2_x} \|\vartheta_{2,R} u(t)\|_{L^2_x} \\
 		&\lesssim R^{-2} \|\nabla_y(\vartheta_{2,R} u(t))\|_{L^2_x}.
 		\end{align*}
 		It follows that
 		\begin{align*}
 		\left|\rea \int \vartheta_{2,R}(y) v(t,x) \overline{u}^2(t,x) dx \right| &\lesssim R^{-1} \|\nabla_y(\vartheta_{2,R} u(t))\|_{L^2_x}^{1/2} \|\partial_4 u(t)\|_{L^2_x}^{1/2} \\
 		&\lesssim R^{-1} \left( \|\nabla_y(\vartheta_{2,R} u(t))\|_{L^2_x} + \|\partial_4 u(t)\|_{L^2_x}\right) \\
 		&\lesssim R^{-1} \left( \|\nabla_y(\vartheta_{2,R} u(t))\|^2_{L^2_x} + \|\partial_4 u(t)\|^2_{L^2_x}  +1 \right) \\
 		&\lesssim R^{-1} \left( \|\vartheta_{2,R} \nabla_y u(t)\|^2_{L^2_x} +  \|\partial_4 u(t)\|^2_{L^2_x} +1 \right),
 		\end{align*}
 		where we have used the conservation of mass and $\|\nabla \vartheta_{2,R}\|_{L^\infty_x}\lesssim 1$ to get the last estimate. Collecting the above estimates, we prove \eqref{loca-viri-est-cyli-2}.
 	\end{proof}

 	\section{Proof of the main results}
	\label{S5}
	\setcounter{equation}{0}
	We are now able to prove the main results stated in Section \ref{S-Main}.
	
	\subsection{The intercritical case} The proof of Theorem \ref{theo-blow-5d} is done by performing an ODE argument, by using the a-priori estimates we proved in the previous Section.\\
	 
	\noindent {\it Proof of Theorem \ref{theo-blow-5d}.}
	Let $(u_0,v_0) \in \Sigma_5 \times \Sigma_5$ satisfy \eqref{cond-blow}. We will show that $T^*<\infty$. Assume by contradiction that $T^*=\infty$. By Lemma \ref{lem-uppe-boun-5d}, there exist positive constants $\vareps$ and $c$ such that
	\begin{align} \label{uppe-boun-5d-app}
	G(u(t),u(t)) + \vareps T(u(t),v(t)) \leq -c
	\end{align}
	for all $t\in [0,\infty)$. On the other hand, by Lemma \ref{lem-loca-viri-est-cyli-1}, we have for all $t\in [0,\infty)$,
	\begin{align} \label{est-viri-cyli-app}
	\frac{d}{dt} M_{\varphi_R}(t)\leq 8G(u(t),v(t)) + CR^{-\frac{3}{2}} \|\nabla u(t)\|^2_{L^2} + o_R(1),
	\end{align}
	where $\varphi_R$ is as in \eqref{defi-varphi-R} and $M_{\varphi_R}(t)$ is as in \eqref{defi-M-varphi}. It follows from \eqref{uppe-boun-5d-app} and \eqref{est-viri-cyli-app} that for all $t\in [0,\infty)$,
	\begin{align*}
	\frac{d}{dt} M_{\varphi_R}(t) \leq -8c - 8\vareps T(u(t),v(t)) +CR^{-\frac{3}{2}} \|\nabla u(t)\|^2_{L^2} + o_R(1).
	\end{align*}
	By choosing $R>1$ sufficiently large, we get
	\begin{align} \label{est-M-varphi-R-app-1}
	\frac{d}{dt}M_{\varphi_R}(t) \leq -4c -4\vareps T(u(t),v(t))
	\end{align}
	for all $t\in [0,\infty)$. Integrating the above inequality, we see that $M_{\varphi_R}(t) <0$ for all $t\geq t_0$  with some $t_0>0$ sufficiently large. We infer from \eqref{est-M-varphi-R-app-1} that
	\begin{align} \label{est-M-varphi-R-app-2}
	M_{\varphi_R}(t) \leq -4\vareps \int_{t_0}^t T(u(s),v(s)) ds
	\end{align}
	for all $t\geq t_0$. On the other hand, by the H\"older's inequality and the conservation of mass, we have
	\begin{align}
	|M_{\varphi_R}(t)| &\leq C \|\nabla \varphi_R\|_{L^\infty} \left( \|\nabla u(t)\|_{L^2} \|u(t)\|_{L^2} + \|\nabla v(t)\|_{L^2} \|v(t)\|_{L^2} \right) \nonumber \\
	&\leq C(\varphi_R, \kappa, M(u_0,v_0)) \sqrt{T(u(t),v(t))}. \label{est-M-varphi-R-app-3}
	\end{align}
	From \eqref{est-M-varphi-R-app-2} and \eqref{est-M-varphi-R-app-3}, we get
	\begin{align} \label{est-M-varphi-R-app-4}
	M_{\varphi_R}(t) \leq -A \int_{t_0}^t |M_{\varphi_R}(s)|^2 ds
	\end{align}
	for all $t\geq t_0$, where $A=A(\vareps, \varphi_R, \kappa, M(u_0,v_0))>0$. Set 
	\begin{align} \label{est-M-varphi-R-app-5}
	z(t):= \int_{t_0}^t |M_{\varphi_R}(s)|^2 ds, \quad t\geq t_0.
	\end{align}
	We see that $z(t)$ is  non-decreasing and non-negative. Moreover, 
	\[
	z'(t) = |M_{\varphi_R}(t)|^2 \geq A^2 z^2(t), \quad \forall t\geq t_0.
	\]
	For $t_1>t_0$, we integrate \eqref{est-M-varphi-R-app-5} over $[t_1,t]$ to obtain
	\[
	z(t) \geq \frac{z(t_1)}{1-A^2z(t_1)(t-t_1)}, \quad \forall t\geq t_1.
	\]
	This shows that $z(t) \rightarrow +\infty$ as $t \nearrow t^*$, where
	\[
	t^*:= t_1 + \frac{1}{A^2 z(t_1)} >t_1.
	\]
	In particular, we have
	\[
	M_{\varphi_R}(t) \leq -Az(t) \rightarrow -\infty
	\]
	as $t\nearrow t^*$. Thus the solution cannot exist for all time $t\geq 0$. Therefore it must blow-up in finite time.
	\hfill $\Box$
		
	\subsection{The energy-critical case} The proof is done by an ODE argument as well, similarly to the intercritical case. \\
	
	\noindent {\it Proof of Theorem \ref{theo-blow-6d}.} The proof is similar to that of Theorem \ref{theo-blow-5d} using \eqref{uppe-boun-6d} and \eqref{loca-viri-est-cyli-1}. Thus we omit the details.
	\hfill $\Box$

	\subsection{The mass-critical case} In this subsection, we give the proofs of the blow-up/grow-up results given in Theorems \ref{theo-blow-grow-4d} and \ref{theo-blow-grow-Sigma-4d}. \\
	
	\noindent {\it Proof of Theorem \ref{theo-blow-grow-4d}.}
	Let $(u_0,v_0) \in H^1\times H^1$ be radially symmetric satisfying $E(u_0,v_0)<0$. Let $(u,v)$ be the corresponding solution to \eqref{SNLS} defined on the maximal forward time interval $[0,T^*)$. If $T^*<\infty$, then we are done. Suppose that $T^*=\infty$. We will show that there exists a constant $C>0$ depending only on $\kappa, M(u_0,v_0)$, and $E(u_0,v_0)$ such that 
	\[ 
	T(u(t),v(t)) \geq Ct^2
	\]
	for all $t\geq t_0$, where $t_0\gg 1$, namely that \eqref{grow-4d} holds true. Let $\varphi_R$ be as in \eqref{defi-varphi-R} and $M_{\varphi_R}(t)$ as in \eqref{defi-M-varphi}. By Lemma \ref{lem-loca-viri-est-rad-2} and the conservation of energy, we have for all $t\in [0,\infty)$,
	\[
	\frac{d}{dt}M_{\varphi_R}(t) \leq 16E(u_0,v_0) - 4 \int \left( \theta_{1,R}(r) - CR^{-\frac{3}{2}} \left(\theta_{2,R}(x)\right)^2\right) |\nabla u(t,x)|^2 dx + o_R(1),
	\]
	where $\theta_{1,R}$ and $\theta_{2,R}$ are as in \eqref{defi-theta-12-R}. \\
	
\noindent	If we can show that 
	\begin{equation} \label{prop-theta-12-R}
	\theta_{1,R}(r) - CR^{-\frac{3}{2}} \left(\theta_{2,R}(x)\right)^2 \geq 0, \quad \forall r=|x|\geq0,
	\end{equation}
	then, by taking $R>1$ large enough, we get
	\[
	\frac{d}{dt} M_{\varphi_R}(t) \leq 8E(u_0,v_0) <0
	\]
	for all $t\in [0,\infty)$. Integrating the above estimate, we have
	\[
	M_{\varphi_R}(t) \leq M_{\varphi_R}(0) + 8E(u_0,v_0) t
	\]
	which implies
	\[
	M_{\varphi_R}(t) \leq 4E(u_0,v_0) t<0
	\]
	for all $t\geq t_0$, where $t_0:= \frac{|M_{\varphi_R}(0)|}{-4E(u_0,v_0)}$. By \eqref{est-M-varphi-R-app-3}, we have
	\[
	- 4E(u_0,v_0) t \leq -M_{\varphi_R}(t) = |M_{\varphi_R}(t)| \leq C(\varphi_R, \kappa, M(u_0,v_0)) \sqrt{T(u(t),v(t))}
	\]
	for all $t\geq t_0$. This shows \eqref{grow-4d}.\\
	
	It remains to find a suitable cut-off function $\chi$ (as defined in \eqref{defi-chi}) so that \eqref{prop-theta-12-R} holds. For the choice of such a function, we are inspired by \cite{OT-JDE}. Let 
	\[
	\zeta(s):= \left\{
	\begin{array}{c l c}
	2s & \text{if}& 0\leq s \leq 1, \\
	2[s-(s-1)^3] &\text{if}& 1<s\leq 1+1/\sqrt{3}, \\
	\zeta'(s) <0 &\text{if}& 1+ 1/\sqrt{3} <s < 2, \\
	0 &\text{if}& s \geq 2,
	\end{array}
	\right.
	\]
	and 
	\begin{align} \label{defi-chi-app}
	\chi(r):= \int_{0}^{r}\zeta(s)ds.
	\end{align}
	It is easy to see that $\chi$ satisfies \eqref{defi-chi}. Define $\varphi_R$ as in \eqref{defi-varphi-R}. We will show that \eqref{prop-theta-12-R} is satisfied for this choice of $\varphi_R$. Indeed, we have $\theta_{1,R}(r) = 2-\varphi''_R(r)$ and 
	\[
	\theta_{2,R}(x) = 8-\Delta \varphi_R(x) = 2- \varphi''_R(r) + 3 \left(2-\frac{\varphi'_R(r)}{r}\right),
	\]
	where the latter follows from the fact that
	\[
	\Delta \varphi_R(x) = \varphi''_R(r) +\frac{3}{r}\varphi'_R(r).
	\]
	We infer, from the definition of $\varphi_R,$ that
	\[
	\varphi'_R(r) = R\chi'(r/R) = R \zeta(r/R), \quad \varphi''_R(r) = \chi''(r/R) = \zeta'(r/R).
	\]
	\begin{itemize}[leftmargin=6mm]
		\item For $0 \leq r =|x|\leq R$, we have $\theta_{1,R}(r) = \theta_{2,R}(x) =0$.
		\item For $R < r = |x| \leq (1+1/\sqrt{3}) R$, we have
		\[
		\theta_{1,R}(r) = 6(r/R-1)^2
		\]
		and 
		\[
		\theta_{2,R}(x) = 2(r/R-1)^2 \left(3+2 \frac{r/R-1}{r/R}\right) < 2\left(3+\frac{2}{\sqrt{3}}\right)(r/R-1)^2.
		\]
		By choosing $R>1$ sufficiently large, we see that \eqref{prop-theta-12-R} is fulfilled. 
		\item When $r> (1+1/\sqrt{3})R$, we see that $\zeta'(r/R) \leq 0$, so $\theta_{1,R}(r) = 2-\varphi''_R(r) \geq 2$. We also have $\theta_{2,R}(r) \leq C$ for some constant $C>0$. Thus by choosing $R>1$ sufficiently large, we have \eqref{prop-theta-12-R}. 
	\end{itemize}
	The proof is complete by glueing together \eqref{grow-4d} and \eqref{prop-theta-12-R}.
	\hfill $\Box$\\
	
	We can now proceed with the proof of the cylindrical case. \\
	
	\noindent {\it Proof of Theorem \ref{theo-blow-grow-Sigma-4d}.}
	Let $(u_0,v_0) \in \Sigma_4 \times \Sigma_4$ satisfy $E(u_0,v_0)<0$. Let $(u,v)$ be the corresponding solution to \eqref{SNLS} defined on the maximal forward time interval $[0,T^*)$.

	\noindent First we consider the non-mass-resonance case, i.e. $0<\kappa \ne \frac{1}{2}$. If $T^*<\infty$, then we are done. Suppose that $T^*=\infty$. We will show that there exists a time sequence $t_n \rightarrow \infty$ such that $\|(u(t_n),v(t_n))\|_{H^1 \times H^1} \rightarrow \infty$ as $n\rightarrow \infty$. Assume by contradiction that it is not true, that is, 
	\begin{align} \label{contra}
	\sup_{t\in[0,\infty)} \|(u(t),v(t))\|_{H^1 \times H^1} \leq C <\infty.
	\end{align}
	Let $\varphi_R$ be as in \eqref{defi-varphi-R-psi} and $M_{\varphi_R}(t)$ as in \eqref{defi-M-varphi}. By Lemma \ref{lem-loca-viri-est-cyli-2} and the conservation of energy, we have for all $t\in [0,\infty)$,
	\begin{align}
	\frac{d}{dt} M_{\varphi_R}(t) &\leq 16 E(u_0,v_0) -4 \int \left(\vartheta_{1,R}(\rho) - CR^{-1} (\vartheta_{2,R}(y))^2\right) |\nabla_y u(t,x)|^2 dx \nonumber \\
	& + CR^{-1} \|\partial_4 u(t)\|^2_{L^2} + o_R(1) \label{est-cyli-4d-app}
	\end{align}
	for some constant $C>0$ depending only on $\kappa$ and $M(u_0,v_0)$, where $\vartheta_{1,R}$ and $\vartheta_{2,R}$ are as in \eqref{defi-vartheta-12-R}. This, together with \eqref{contra}, gives 
	\begin{align*}
	\frac{d}{dt} M_{\varphi_R}(t) &\leq 16 E(u_0,v_0) -4 \int \left(\vartheta_{1,R}(\rho) - CR^{-1} (\vartheta_{2,R}(y))^2\right) |\nabla_y u(t,x)|^2 dx \\
	& + CR^{-1} + o_R(1).
	\end{align*}
	for all $t\in [0,\infty).$
	\\
	
	\noindent Provided that we prove  
	\begin{align} \label{prop-vartheta-12-R}
	\vartheta_{1,R}(\rho) - CR^{-1} (\vartheta_{2,R}(y))^2 \geq 0, \quad \forall \rho=|y| \geq 0,
	\end{align}
	then we can choose $R>1$ large enough so that for all $t\in [0,\infty)$,
	\[
	\frac{d}{dt}M_{\varphi_R}(t) \leq 4E(u_0,v_0).
	\]
	Arguing as in the proof of Theorem \ref{theo-blow-grow-4d}, we have
	\[
	-4E(u_0,v_0) t \leq C(\varphi_R, \kappa, M(u_0,v_0)) \sqrt{T(u(t),v(t))}
	\]
	for all $t\geq t_0$, where $t_0 \gg 1$. In particular, we have
	\[
	T(u(t),v(t)) \geq C(\varphi_R, \kappa, M(u_0,v_0),E(u_0,v_0)) t^2
	\]
	for all $t\geq t_0$ which contradicts \eqref{contra} for $t$ sufficiently large. \\
	
	Next we consider the mass-resonance case, i.e. $\kappa=\frac{1}{2}$. If $T^*<\infty$, then we are done. Suppose that $T^*=\infty$. We will show that there exists a time sequence $t_n \rightarrow \infty$ such that $\|\partial_4 u(t_n)\|_{L^2}\rightarrow \infty$ as $n\rightarrow \infty$. Assume by contradiction that it does not hold, i.e. 
	\begin{align*}
	\sup_{t\in [0,\infty)} \|\partial_4 u(t)\|_{L^2} \leq C<\infty. 
	\end{align*}
	Thanks to \eqref{est-cyli-4d-app}, we have for all $t\in [0,\infty)$,
	\begin{align*}
	\frac{d}{dt} M_{\varphi_R}(t) &\leq 16 E(u_0,v_0) -4 \int \left(\vartheta_{1,R}(\rho) - CR^{-1} (\vartheta_{2,R}(y))^2\right) |\nabla_y u(t,x)|^2 dx \\
	&+ CR^{-1} + o_R(1).
	\end{align*}
	Provided that \eqref{prop-vartheta-12-R} holds true, we can choose $R>1$ sufficiently large to get for all $t\in [0,\infty)$,
	\[
	\frac{d}{dt} M_{\varphi_R} (t) \leq 4 E(u_0,v_0).
	\]
	As $\kappa=\frac{1}{2}$, we see that
	\[
	\frac{d}{dt} V_{\varphi_R}(t) = M_{\varphi_R}(t),
	\]
	where
	\[
	V_{\varphi_R}(t) := \int \varphi_R (x) \left(|u(t,x)|^2 +2|v(t,x)|^2 \right) dx.
	\]
	It follows that, for all $t\in [0,\infty)$,
	\[
	\frac{d^2}{dt^2} V_{\varphi_R}(t) \leq 4 E(u_0,v_0) <0.
	\]
	Integrating the above inequality, there exists $t_0>0$ sufficiently large so that $V_{\varphi_R}(t_0)<0$ which is impossible. \\
	
	Finally, let us choose a suitable cut-off function $\varphi_R$ so that \eqref{prop-vartheta-12-R} is fulfilled. Let $\chi$ be as in \eqref{defi-chi-app}. It is easy to see that $\chi$ satisfies \eqref{defi-chi}. Define
	\[
	\psi_R(y):= \psi_R(\rho) = R^2 \chi(\rho/R), \quad \rho=|y|
	\]
	and let $\varphi_R$ be as in \eqref{defi-varphi-R-psi}, namely $\varphi_R(x)=\psi_R(y)+x_4^2$. For $\vartheta_{1,R}(\rho) = 2-\psi''_R(\rho)$ and
	\[
	\vartheta_{2,R}(y) = 6-\Delta_y \psi_R(y) = 2- \psi''_R(\rho) + 2 \left(2-\frac{\psi'_R(\rho)}{\rho}\right)
	\]
	since $\Delta_y \psi_R(y) = \psi''_R(\rho) +\frac{2}{\rho}\psi'_R(\rho)$,	we can infer that \eqref{prop-vartheta-12-R} is satisfied for this choice of $\psi_R$. The proof is similar to the one for \eqref{prop-theta-12-R}, so we omit the details.	The proof is complete.
	\hfill $\Box$
	
\section*{Acknowledgements}
V. D. D.  was supported in part by the Labex CEMPI (ANR-11-LABX-0007-01).  L.F. was supported by the EPSRC New Investigator Award (grant no. EP/S033157/1).



\appendix
	
\section{Localized virial estimates for the 4D general quadratic system}
\label{S6}
\setcounter{equation}{0}
In this appendix, we provide some localized virial estimates related to a generalized system of NLS with quadratic interactions \eqref{GQNLS} in the mass-critical case $d=4$.

\begin{lemma} \label{lem-viri-est-rad-4d-GQNLS}
	Let $d=4$ and $\vec u$ be a radial $H^1$-solution to \eqref{GQNLS} defined on the maximal forward time interval $[0,T^*)$. Let $\varphi_R$ be as in \eqref{defi-varphi-R} and denote
	\begin{align} \label{defi-M-varphi-GQNLS}
	\mathcal M_{\varphi_R}(t):= 2 \ima \sum_{j=1}^N a_j \int \nabla \varphi_R(x) \cdot \nabla u_j(t,x) \overline{u}_j(t,x) dx.
	\end{align}
	Then we have for all $t\in [0,T^*)$,
	\begin{equation*}
	\frac{d}{dt} \mathcal M_{\varphi_R}(t) \leq 16 \mathcal E (\vec u(t)) - 4 \int \left( \theta_{1,R}(r) - CR^{-\frac{3}{2}} \left(\theta_{2,R}(x)\right)^2 \right) \left( \sum_{j=1}^N a_j |\nabla u_j(t,x)|^2\right) dx + o_R(1)
	\end{equation*}
	for some constant $C>0$ depending only on $\kappa$, $\boldsymbol{a}=(a_1,\dots, a_N)$, and $\mathcal M(\vec u_0)$, where $\theta_{1,R}$ and $\theta_{2,R}$ are as in \eqref{defi-theta-12-R}.
\end{lemma}

\begin{proof}
	Arguing as in the proof of \cite[Theorem 4.1]{NP-blow}, we have for all $t\in [0,T^*)$,
	\begin{align*}
	\frac{d}{dt} \Mcal_{\varphi_R}(t) &= 16 \mathcal E(\vec u(t)) - 4 \int (2-\varphi''_R) \left(\sum_{j=1}^N a_j |\nabla u_j(t)|^2 \right) dx \\
	&- \int \Delta^2 \varphi_R \left( \sum_{j=1}^N a_j |u_j(t)|^2\right) dx + 2 \rea \int (8-\Delta \varphi_R) F(\vec u(t)) dx.
	\end{align*}
	By the conservation of mass, we have
	\[
	\left|\int \Delta^2 \varphi_R \left( \sum_{j=1}^N a_j |u_j(t)|^2\right) dx\right| \lesssim R^{-2}.
	\]
	Thus we get
	\begin{equation} \label{est-1}
	\frac{d}{dt} \Mcal_{\varphi_R}(t) \leq 16 \mathcal E(\vec u(t)) - 4\int \theta_{1,R} \left(\sum_{j=1}^N a_j |\nabla u_j(t)|^2 \right) dx + 2\rea \int \theta_{2,R} F(\vec u(t))dx + CR^{-2},
	\end{equation}
	where $\theta_{1,R}$ and $\theta_{2,R}$ are as in \eqref{defi-theta-12-R}. By the assumption (H6) in \cite{NP-blow}, the radial Sobolev embedding and the conservation of mass, we estimate
	\begin{equation*}
	\begin{aligned}
	\left|\rea \int \theta_{2,R} F(\vec u(t))dx \right| &\leq \int \theta_{2,R} |F(\vec u(t))| dx \\
	&\leq \int \theta_{2,R} \left( \sum_{j=1}^N |u_j(t)|^3\right) dx \\
	&\leq \sum_{j=1}^N \sup_{|x|\geq R} |\theta_{2,R}(x) u_j(t,x)| \|u_j(t)\|_{L^2}^2 \\
	&\lesssim  R^{-\frac{3}{2}} \sum_{j=1}^N \|\nabla(\theta_{2,R} u_j(t))\|^{\frac{1}{2}}_{L^2} \|\theta_{2,R} u_j(t)\|^{\frac{1}{2}}_{L^2} \|u_j(t)\|^2_{L^2}\\
	&\lesssim R^{-\frac{3}{2}} \sum_{j=1}^N  \|\nabla(\theta_{2,R} u_j(t))\|^{\frac{1}{2}}_{L^2}.
	\end{aligned}
	\end{equation*}
	Thanks to the conservation of mass and the fact that $\|\nabla \theta_{2,R}\|_{L^\infty} \lesssim 1$, we have
	\[
	\|\nabla (\theta_{2,R} u_j(t))\|_{L^2} \lesssim \|\theta_{2,R} \nabla_j u(t)\|_{L^2} +1
	\]
	which implies that
	\begin{equation} \label{est-2}
	\begin{aligned}
	\left|\rea \int \theta_{2,R} F(\vec u(t))dx \right| &\lesssim R^{-\frac{3}{2}} \sum_{j=1}^N \left( \|\theta_{2,R} \nabla u_j(t)\|^2_{L^2} +1 \right)  \\
	&\lesssim R^{-\frac{3}{2}} \int \left(\theta_{2,R} \right)^2 \left( \sum_{j=1}^N a_j |\nabla u_j(t)|^2\right) dx + R^{-\frac{3}{2}}.
	\end{aligned}
	\end{equation}
	Collecting \eqref{est-1} and \eqref{est-2}, we finish the proof.
\end{proof}

\begin{lemma} \label{lem-viri-est-cyli-4d-GQNLS}
	Let $d=4$ and $\vec u$ be a $\Sigma_4$-solution to \eqref{GQNLS} defined on the maximal forward time interval $[0,T^*)$. Let $\varphi_R$ be as in \eqref{defi-varphi-R-psi} and denote $\mathcal M_{\varphi_R}(t)$ as in \eqref{defi-M-varphi-GQNLS}. Then we have for all $t\in [0,T^*)$,
	\begin{equation*}
	\begin{aligned} 
	\frac{d}{dt} \mathcal M_{\varphi_R}(t) &\leq 16 \mathcal E(\vec u(t)) -4 \int \left(\vartheta_{1,R}(\rho) - CR^{-1} (\vartheta_{2,R}(y))^2\right) \left(\sum_{j=1}^N a_j |\nabla_y u_j(t,x)|^2\right) dx  \\
	&+ CR^{-1} \sum_{j=1}^N\|\partial_4 u_j(t)\|^2_{L^2} + o_R(1)
	\end{aligned}
	\end{equation*}
	for some constant $C>0$ depending only on $\kappa$, $\boldsymbol{a}=(a_1,\dots,a_N)$, and $\mathcal M(\vec u_0)$, where $\vartheta_{1,R}$ and $\vartheta_{2,R}$ are as in \eqref{defi-vartheta-12-R}.
\end{lemma}

\begin{proof}
	Using localized virial identities similar to Item (4) of Remark \ref{rem-viri} (see also Lemma \ref{lem-loca-viri-est-cyli-2}), we have for all $t\in [0,T^*)$,
	\begin{equation*}
	\begin{aligned}
	\frac{d}{dt} \mathcal M_{\varphi_R}(t) &\leq 16 \mathcal E(\vec u(t)) + CR^{-2} - 4 \int \vartheta_{1,R}(\rho) \left(\sum_{j=1}^N a_j |\nabla_y u_j(t,x)|^2\right) dx \\
	&+ 2\rea \int \vartheta_{2,R}(y) F(\vec u(t,x)) dx,
	\end{aligned}
	\end{equation*}
	where $\vartheta_{1,R}$ and $\vartheta_{2,R}$ are as in \eqref{defi-vartheta-12-R}. We estimate
	\begin{equation*}
	\begin{aligned}
	\left| \rea \int \vartheta_{2,R}(y) F(\vec u(t,x)) dx\right| &\leq \int \vartheta_{2,R}(y) |F(\vec u(t,x))| dx \\
	&\leq \int \vartheta_{2,R}(y) \left(\sum_{j=1}^N |u_j(t,x)|^3 \right) dx \\
	&\leq \sum_{j=1}^N \left( \int \left(\vartheta_{2,R}(y)\right)^2 |u_j(t,x)|^4 dx\right)^{1/2} \|u_j(t)\|_{L^2_x} \\
	&\lesssim \sum_{j=1}^N \left( \int \left(\vartheta_{2,R}(y)\right)^2 |u_j(t,x)|^4 dx\right)^{1/2}.
	\end{aligned}
	\end{equation*}
	Estimating as in Lemmas \ref{lem-loca-viri-est-cyli-1} and \ref{lem-loca-viri-est-cyli-2}, we have
	\begin{align*}
	\left(\int \left(\vartheta_{2,R}(y)\right)^2 |u_j(t,x)|^4 dx \right)^{1/2} &\lesssim R^{-1} \|\nabla_y(\vartheta_{2,R} u_j(t))\|^{1/2}_{L^2_x} \|\partial_4 u_j(t)\|^{1/2}_{L^2_x}\\
	&\lesssim R^{-1} \left( \|\nabla_y(\vartheta_{2,R} u_j(t)\|_{L^2_x} + \|\partial_4 u_j(t)\|_{L^2_x}\right) \\
	&\lesssim R^{-1} \left(\|\nabla_y(\vartheta_{2,R} u_j(t))\|^2_{L^2_x} + \|\partial_4 u_j(t)\|^2_{L^2_x} +1 \right) \\
	&\lesssim R^{-1} \left(\|\vartheta_{2,R} \nabla_y u_j(t)\|^2_{L^2_x} + \|\partial_4 u_j(t)\|^2_{L^2_x} +1 \right).	
	\end{align*}
	The proof is complete by collecting the above estimates.
\end{proof}



\begin{bibdiv}
\begin{biblist}

\bib{AVDF}{article}{
author={Ardila, A. H.},
author={Dinh, V. D.},
author={Forcella, L.},
title={Sharp conditions for
scattering and blow-up for a system of NLS arising in optical materials with $\chi^3$ nonlinear response},
journal={Communications in Partial Differential Equations, to appear}
eprint={https://arxiv.org/abs/2009.05933},
}

			
\bib{BF19}{article}{
author={Bellazzini, J.},
author={Forcella, L.},
title={Asymptotic dynamic for dipolar quantum gases below the ground state energy threshold},
journal={J. Funct. Anal.},
volume={277},
date={2019},
number={6},
pages={1958--1998},
issn={0022-1236},
}

\bib{BF20}{article}{
author={Bellazzini, {J.}},
author={Forcella, {L.}},
title={Dynamical collapse of cylindrical symmetric dipolar Bose-Einstein Condensates},
journal={preprint},
eprint={https://arxiv.org/abs/2005.02894}, 
}

\bib{Cazenave}{book}{
author={Cazenave, T.},
title={Semilinear Schr\"{o}dinger equations},
series={Courant Lecture Notes in Mathematics},
volume={10},
publisher={New York University, Courant Institute of Mathematical Sciences, New York; American Mathematical Society, Providence, RI},
date={2003},
pages={xiv+323},
isbn={0-8218-3399-5},
}

\bib{CO}{article}{
author={Cho, Y.},
author={Ozawa, T.},
title={Sobolev inequalities with symmetry},
journal={Commun. Contemp. Math.},
volume={11},
date={2009},
number={3},
pages={355--365},
issn={0219-1997},
}
		
\bib{CCO}{article}{
author={Colin, M.},
author={Colin, Th.},
author={Ohta, M.},
title={Stability of solitary waves for a system of nonlinear Schr\"{o}dinger
equations with three wave interaction},
language={English, with English and French summaries},
journal={Ann. Inst. H. Poincar\'{e} Anal. Non Lin\'{e}aire},
volume={26},
date={2009},
number={6},
pages={2211--2226},
issn={0294-1449},
}

\bib{CdMS}{article}{
author={Colin, M.},
author={Di Menza, L.},
author={Saut, J. C.},
title={Solitons in quadratic media},
journal={Nonlinearity},
volume={29},
date={2016},
number={3},
pages={1000--1035},
issn={0951-7715},
}

\bib{Dinh-NA}{article}{
author={Dinh, V. D.},	
title={Existence, stability of standing waves and the characterization of finite time blow-up solutions for a system NLS with quadratic interaction},
journal={Nonlinear Anal.},
volume={190},	
date={2020},
pages={111589, 39},
issn={0362-546X},
}

\bib{Dinh-insta}{article}{
	author={Dinh, {V. D.}},
	title={Strong Instability of Standing Waves for a System NLS with
		Quadratic Interaction},
	journal={Acta Math. Sci. Ser. B (Engl. Ed.)},
	volume={40},
	date={2020},
	number={2},
	pages={515--528},
	issn={0252-9602},
}

\bib{DFH}{article}{
author={Dinh, V. D.},
author={Forcella, L.},
author={Hajaiej, H.},
title={Mass-Energy threshold dynamics for dipolar Quantum Gases},
journal={Communications in Mathematical Sciences, to appear}
eprint={https://arxiv.org/abs/2009.05933},
}

\bib{DM-MRL}{article}{
author={Dodson, B.},
author={Murphy, J.},
title={A new proof of scattering below the ground state for the non-radial focusing NLS},
journal={Math. Res. Lett.},
volume={25},
date={2018},
number={6},
pages={1805--1825},
issn={1073-2780},
}
			
			
		
\bib{Hamano}{article}{
author={Hamano, M.},
title={Global dynamics below the ground state for the quadratic Sch\"odinger system in 5d},
journal={preprint},
eprint={https://arxiv.org/abs/1805.12245},
}
			
\bib{HIN}{article}{
author={Hamano, M.},
author={Inui, T.},
author={Nishimura, K.},
title={Scattering for the quadratic nonlinear Schr\"odinger system in $\mathbb R^5$ without mass-resonance condition},
journal={preprint},
eprint={https://arxiv.org/abs/1903.05880},
}
		
\bib{HOT}{article}{
author={Hayashi, N.},
author={Ozawa, T.},
author={Tanaka, K.},
title={On a system of nonlinear Schr\"{o}dinger equations with quadratic interaction},
journal={Ann. Inst. H. Poincar\'{e} Anal. Non Lin\'{e}aire},
volume={30},
date={2013},
number={4},
pages={661--690},
issn={0294-1449},
}

\bib{HR-CPDE}{article}{
author={Holmer, J.},
author={Roudenko, S.},
title={Divergence of infinite-variance nonradial solutions to the 3D NLS
equation},
journal={Comm. Partial Differential Equations},
volume={35},
date={2010},
number={5},
pages={878--905},
issn={0360-5302},
}	

\bib{Inui1}{article}{
author={Inui, T.},
title={Global dynamics of solutions with group invariance for the
nonlinear Schr\"{o}dinger equation},
journal={Commun. Pure Appl. Anal.},
volume={16},
date={2017},
number={2},
pages={557--590},
issn={1534-0392},
}
		
\bib{Inui2}{article}{
author={Inui, {T.}},
title={Remarks on the global dynamics for solutions with an infinite group invariance to the nonlinear Schr\"{o}dinger equation},
conference={
title={Harmonic analysis and nonlinear partial differential equations},},
book={series={RIMS K\^{o}ky\^{u}roku Bessatsu, B70},
publisher={Res. Inst. Math. Sci. (RIMS), Kyoto},},
date={2018},
pages={1--32},
}

\bib{IKN-mass}{article}{
author={Inui, T.},
author={Kishimoto, N.},
author={Nishimura, K.},
title={Scattering for a mass critical NLS system below the ground state with and without mass-resonance condition},
journal={Discrete Contin. Dyn. Syst.},
volume={39},
date={2019},
number={11},
pages={6299--6353},
issn={1078-0947},
}

\bib{IKN-NA}{article}{
author={Inui, {T.}},
author={Kishimoto, {N.}},
author={Nishimura, {K.}},
title={Blow-up of the radially symmetric solutions for the quadratic nonlinear Schr\"{o}dinger system without mass-resonance},
journal={Nonlinear Anal.},
volume={198},
date={2020},
pages={111895, 10},
issn={0362-546X},
}
							
\bib{KM}{article}{
author={Kenig, C. E.},
author={Merle, F.},
title={Global well-posedness, scattering and blow-up for the energy-critical, focusing, nonlinear Schr\"odinger equation in the radial case},
journal={Invent. Math.},
volume={166},
date={2006},
number={3},
pages={645--675},
issn={0020-9910},
}

\bib{Kiv}{article}{
author={Kivshar, Y. S.}, 
author={Sukhorukova, A. A.}, 
author={Ostrovskayaa, E. A.}, 
author={Alexandera, T. J.}, 
author={Bang, O.}, 
author={Saltiel, S. M.}, 
author={Clausen, C. B.}, 
author={Christiansen, P. L.},
title={Multi-component optical solitary waves},
journal={Physica A: Statistical Mechanics and its Applications},
volume={288},
number={1--4},
year={2000},
pages={152--173},
}

\bib{KS}{article}{
author={Koynov, K.},
author={Saltiel, S.},
title={Nonlinear phase shift via multistep $\chi^{(2)}$ cascading},
journal={Optics Communications},
volume={152},
number={1},
pages={96--100},
year={1998},
}

\bib{Mar}{article}{
author={Martel, Y.},
title={Blow-up for the nonlinear Schr\"{o}dinger equation in nonisotropic spaces},
journal={Nonlinear Anal.},
volume={28},
date={1997},
number={12},
pages={1903--1908},
issn={0362-546X},
}

\bib{NP-CCM}{article}{
author={Noguera, N.},
author={Pastor, {A.}},
title = {A system of Schr\"odinger equations with general quadratic-type nonlinearities},
journal = {Communications in Contemporary Mathematics (in press)},
year = {2020},
eprint = {https:doi.org/10.1142/S0219199720500236},
}

\bib{NP-blow}{article}{
author={Noguera, N.},
author={Pastor, A.},
title={Blow-up solutions for a system of Schr\"odinger equations with general quadratic type nonlinearities in dimensions five and six},
journal={preprint},
eprint={https://arxiv.org/abs/2003.11103},
}
			
\bib{NP-DPDE}{article}{
author={Noguera , N.},
author={Pastor, A.},
title={On the dynamics of a quadratic Schr\"odinger system in dimension $n =5$},
journal={Dynamics of Partial Differential Equations}
year={2020},
volume={17},
number={1},
pages={1-17},
}
			
\bib{OT-JDE}{article}{
author={Ogawa , T.},
author={Tsutsumi, Y.},
title={Blow-up of $H^1$ solution for the nonlinear Schr\"{o}dinger equation},
journal={J. Differential Equations},
volume={92},
date={1991},
number={2},
pages={317--330},
issn={0022-0396},
}
			
			
			
\bib{WY}{article}{
author={Wang, H.},
author={Yang, Q.},
title={Scattering for the 5D quadratic NLS system without mass-resonance},
journal={J. Math. Phys.},
volume={60},
date={2019},
number={12},
pages={121508, 23},
issn={0022-2488},
}
			
\bib{Yoshida}{article}{
author={Yoshida, N.},
title={Master Thesis},
journal={Osaka University},
year={2013},
}
		
\end{biblist}
\end{bibdiv}
\end{document}